\documentclass[12pt, twoside]{article}
\usepackage{amsmath,amsthm,amssymb}
\usepackage{times}
\usepackage{enumerate}

\def\titlerunning#1{\gdef\titrun{#1}}
\makeatletter
\def\author#1{\gdef\autrun{\def\and{\unskip, }#1}\gdef\@author{#1}}
\def\address#1{{\def\and{\\\hspace*{18pt}}\renewcommand{\thefootnote}{}%
\footnote {#1}}%
\markboth{\autrun}{\titrun}}
\makeatother
\def\email#1{e-mail: #1}
\def\subjclass#1{{\renewcommand{\thefootnote}{}%
\footnote{\emph{Mathematics Subject Classification (2010):} #1}}}
\def\keywords#1{\par\medskip
\noindent\textbf{Keywords.} #1}

\newtheorem{thm}{Theorem}[section]
\newtheorem{cor}[thm]{Corollary}
\newtheorem{lem}[thm]{Lemma}

\newenvironment{dedication}
  {\thispagestyle{empty}
   \itshape             
   \raggedleft          
  }
  {\par 
  }

\theoremstyle{definition}

\newtheorem*{xrem}{Remark}

\numberwithin{equation}{section}

\newcommand{\R}{\mathbf{R}}
\newcommand{\N}{\mathbf{N}}
\newcommand{\C}{\mathbf{C}}
\newcommand{\Z}{\mathbf{Z}}

\newcommand{\Mod}[1]{\ (\textup{mod}\ #1)}

\providecommand{\Var}{\operatorname{Var}}
\providecommand{\sym}{\operatorname{sym}}
\DeclareMathOperator{\res}{res}
\DeclareMathOperator{\arcsinh}{arcsinh}
\DeclareMathOperator{\artanh}{atanh}

\begin{document}


\baselineskip=17pt


\titlerunning{Moments of $L$-functions and Liouville-Green method}

\title{Moments of $L$-functions and Liouville-Green method}

\author{Olga Balkanova
\and 
Dmitry Frolenkov}

\date{}

\maketitle

\address{O. Balkanova: Department of Mathematical Sciences, Chalmers University of Technology and University of Gothenburg,  Chalmers tv\"{a}rgata 3,  Gothenburg 412 96,
   Sweden; \email{olgabalkanova@gmail.com}
\and
D. Frolenkov: Khabarovsk Division of the Institute for Applied Mathematics, Far Eastern Branch, Russian Academy of Sciences and Steklov Mathematical Institute of Russian Academy of
Sciences, 8 Gubkina st., Moscow, 119991, Russia; \email{frolenkov@mi.ras.ru}}


\begin{dedication}
To the memory of Professor N.V. Kuznetsov
\end{dedication}
\begin{abstract}
We show that the percentage of primitive forms of level one and weight $4k\rightarrow \infty$, $k \in \N $ for which the associated $L$-function at the central point is no less than $(\log{k})^{-2}$ is at least $20\%$. The key ingredients of our proof are the Kuznetsov convolution formula and the Liouville-Green method.

\end{abstract}

\section{Introduction}

Non-vanishing results for central values of $L$-functions in families have numerous applications, discovered, for example, in \cite{E, IS, KM, KMV, V}. In particular, this paper is inspired by the work of Iwaniec and Sarnak \cite{IS}, where  the problem of non-existence of Landau-Siegel zeros was approached by studying the non-vanishing of automorphic $L$-functions at the critical point.

In the weight aspect, Iwaniec and Sarnak proved the following result. Let $H_{2k}(1)$ be the space of primitive forms of level $1$ and weight $2k \geq 12$, $k \in \N$. For  $f \in H_{2k}(1)$, let $L_f(1/2)$ be the associated $L$-function at the critical point.
 Then for any $\epsilon>0$ one has
\begin{equation}\label{eq:onaverage}
 \lim_{K\rightarrow \infty}\frac{1}{K}\sum_{k \leq K}
\frac{\#\{f \in H_{4k}(1), \quad L_f(1/2)\geq 1/(\log{k})^2\}}{\#\{f \in H_{4k}(1)\}}
\geq \frac{1}{2}-\epsilon.
\end{equation}

Moreover, it was shown in \cite{IS} that the non-existence of Landau-Siegel zeros for Dirichlet $L$-functions of real primitive characters would follow if  \eqref{eq:onaverage}  is established with a proportion strictly greater than $1/2$.

The problem of non-vanishing for the individual weight was first studied by Iwaniec, Luo and Sarnak \cite{ILS}.
More precisely, \cite[Corollary 1.6]{ILS} states that under the generalized Riemann Hypothesis the percentage of non-vanishing central $L$-values in this case is at least $25\%$.
The first uncondional results for the individual weight were obtained independently by Fomenko \cite{F} and by Lau \& Tsang \cite{LT} who proved that the proportion of non-vanishing  is at least $1/\log{k}$ by computing the asymptotics of pure unmollified first and second moments. Recently, Luo \cite{L} showed that there is a strictly positive proportion of non-vanishing. However, his approach does not allow finding the exact proportion. The reason is that only an upper bound for the mollified second moment was proved in \cite{L}, while the full asymptotic expansion is required to make this result quantitive.

The aim of the present paper is to obtain an effective and strictly positive lower bound on the proportion of non-vanishing central $L$-values.
\begin{thm}\label{main theorem}
For any $\epsilon>0$ there exists $k_0=k_0(\epsilon)$ such that for any $k\geq k_0$ and $k \equiv 0\Mod{2}$ we have
\begin{equation}
\frac{1}{|H_{2k}(1)|}\sum_{\substack{f \in H_{2k}(1)\\ L_{f}(1/2)\geq (\log{k})^{-2}}}1 \geq \frac{1}{5}-\epsilon.
\end{equation}
\end{thm}

The proof of Theorem  \eqref{main theorem} relies in most part on methods developed by Kuznetsov in $1990$'s.
More precisely, we use Theorem \ref{thm:kuznetsov}, which states an exact formula for the second moment of cusp form $L$-functions in the critical strip.
Off-diagonal terms in this formula are given by two shifted convolution sums, namely
\begin{equation*}
\frac{1}{2\sqrt{l}}\sum_{n=1}^{l-1}\tau(n)\tau(l-n)\phi_k\left( \frac{n}{l}\right)+
\frac{1}{\sqrt{l}}\sum_{n=1}^{\infty}\tau(n)\tau(n+l)\Phi_k\left( \frac{l}{n+l}\right),
\end{equation*}
where $\tau(n)$ is the number of divisors function and
$\phi_k(x)$, $\Phi_k(x)$ are certain special functions that can be expressed in terms of
the Gauss hypergeometric function ${}F(a,b,c;x)$.

The most challenging problem is to prove sharp estimates for $\phi_k(x)$ and $\Phi_k(x)$ that are uniform in both parameters $x$ and $k$. The key observation is that these functions are solutions of second-order differential equations,  and therefore, we can choose the Liouville-Green method as our main tool. This method, also called the WKB approximation or the Liouville-Steklov method, is one of the oldest approximation techniques widely applied, for example, in quantum mechanics. The idea of using it in analytic number theory belongs to Kuznetsov.
The method is based on the observation that "close" differential equations have "close" solutions.
Accordingly, in Section \ref{section:LG} we find differential equations satisfied by the functions $\phi_k(x)$ and $\Phi_k(x)$.  The given equations can be approximated by other differential equations which have "simpler" functions as solutions. This allows approximating the off-diagonal terms uniformly in $k$ with any power of precision. See Theorems \ref{thm:approxphi} and \ref{thm:approxPhi}.

More precisely, the Liouville-Green approximation models the behaviour of the functions $\phi_k(x)$ and $\Phi_k(x)$  using the $J$, $Y$ and $K$- Bessel functions with the large parameter $k$ in the argument. Accordingly, we conclude that in the required ranges $\Phi_k(x)$ decays exponentially and $\phi_k(x)$ is oscillatory. To smooth out the oscillations  of $\phi_k(x)$ one can average the off-diagonal terms over weight with a suitable test function, reproving the result  of Iwaniec and Sarnak \eqref{eq:onaverage}, as shown in Theorem \ref{thm: on average}.

Asymptotic formulas for twisted moments have several other applications. For example, Hough \cite{H} considered zero-density estimates for $L$-functions in the weight aspect. His proof is based on the asymptotic evaluation of the second moment near the critical line with  the error term estimated as $O(l^{3/4}k^{-1/2+\epsilon})$ at the central point. The same error bound was obtained by Ng \cite{NMH} using a different approach. Our method (see Theorem \ref{secondmoment}) yields $O(l^{1/2}k^{-1/2+\epsilon})$.

Finally, the techniques developed in the present paper  can be beneficial in solving other problems in analytic number theory which involve analysis of special functions. In particular, our approach yields new results for moments of symmetric square $L$-functions in the weight aspect. See \cite{BF2} for details.

\section{Notation and technical lemmas}\label{section2}

For $ v \in \C $ let
\begin{equation}
\tau_v(n)=\sum_{n_1n_2=n}\left( \frac{n_1}{n_2}\right)^v=n^{-v}\sigma_{2v}(n),
\end{equation}
where
\begin{equation}\sigma_v(n)=\sum_{d|n}d^v.
\end{equation}
Note that $\tau_v(n)=\tau_{-v}(n)$.

Let $e(x)=\exp(2\pi ix)$. The classical Kloosterman sum
\begin{equation*}
S(n,m;c)=\sum_{\substack{a\pmod{c}\\ (a,c)=1}}e\left( \frac{an+a^*m}{c}\right), \quad aa^*\equiv 1\pmod{c},
\end{equation*}
satisfies Weil's bound (see \cite[Theorem 4.5]{Iw})
\begin{equation}
|S(m,n;c)|\leq \tau_0(c)\sqrt{(m,n,c)}\sqrt{c}.
\end{equation}
For $\Re{s}>1$, $m\geq 1$ (see \cite[Eq.~1.5.4]{T})
\begin{equation}\label{ramid}
\sum_{c=1}^{\infty}\frac{S(0,m;c)}{c^{s}}=\frac{\sigma_{1-s}(m)}{\zeta(s)},
\end{equation}
where $\zeta(s)$ is the Riemann zeta function.

Let $H_{2k}(1)$ be the set of primitive forms of level $1$ and weight $2k \geq 12$.
Every $f \in H_{2k}(1)$ has a Fourier expansion of the form
\begin{equation}
f(z)=\sum_{n\geq 1}\lambda_f(n)n^{(2k-1)/2}e(nz).
\end{equation}
 The Fourier coefficients of primitive forms are multiplicative
\begin{equation}\label{eq:mult}
\lambda_f(m)\lambda_f(n)=\sum_{d|(m,n)}\lambda_f\left( \frac{mn}{d^2}\right).
\end{equation}
For  each $f \in H_{2k}(1)$, the associated $L$-function is defined by
\begin{equation}L_f(s)=\sum_{n \geq 1}\frac{\lambda_f(n)}{n^s}, \quad \Re{s}>1
\end{equation}
and the associated symmetric square $L$-function is given by
\begin{equation}\label{sym2 def}
L(\sym^2f,s)=\zeta(2s)\sum_{n=1}^{\infty}\frac{\lambda_f(n^2)}{n^s}=
\sum_{n=1}^{\infty}\frac{\rho_f(n)}{n^s}, \quad \Re{s}>1.
\end{equation}
As a consequence of the relation \eqref{eq:mult}, for $Re{u}>1/2$, $\Re{v}=0$ we have
\begin{equation}\label{Lprod}
\sum_{n=1}^{\infty}\frac{\tau_v(n)\lambda_f(n)}{n^{1/2+u}}=\frac{1}{\zeta(1+2u)}L_f(1/2+u+v)L_f(1/2+u-v).
\end{equation}
Let $\Gamma(s)$ be the Gamma function. The completed $L$-function
\begin{equation}
\Lambda_f(s)=\left(\frac{1}{2 \pi}\right)^s\Gamma\left(s+\frac{2k-1}{2}\right)L_f(s)
\end{equation} satisfies the functional equation
\begin{equation} \label{eq: functionalE}
\Lambda_f(s)=\epsilon_f\Lambda_f(1-s), \quad \epsilon_f=i^{2k}
\end{equation}
and can be analytically continued on the whole complex plane.
It follows from the equation \eqref{eq: functionalE} that $L_f(1/2)=0$ for odd $k$.

The harmonic weight is defined by (see \cite[Lemma 2.5]{ILS})
\begin{equation}\label{harmonic weight}
\omega_f=\frac{ \Gamma(2k-1)}{(4\pi )^{2k-1}\langle f,f\rangle_1}=
\frac{12\zeta(2)}{(2k-1)L(\sym^2f,1)},
\end{equation}
where $\langle f,f\rangle_1$ is the Petersson inner product on the space of level $1$ holomorphic modular forms.
Then the harmonic summation can be written as
\begin{equation}
\sum_{f \in H_{2k}(1)}^{h}\alpha_f:=\sum_{f \in H_{2k}(1)}\omega_f\alpha_f.
\end{equation}

We denote by $J_v(x)$, $Y_v(x)$, $K_v(x)$  the Bessel functions.

\begin{thm}(Petersson's trace formula, \cite[Theorem 4.5]{Iw}) For $2k \geq 12$  and integral $l,n \geq 1$ the following formula holds
\begin{equation}\label{Pet}
\sum_{f \in H_{2k}(1)}^{h}\lambda_f(l)\lambda_f(n)=\delta_{l,n}+2\pi i^{2k}\sum_{c =1}^{\infty}\frac{S(l,n;c)}{c}J_{2k-1}\left( \frac{4\pi \sqrt{ln}}{c}\right).
\end{equation}
\end{thm}


 Consider the Bessel kernels
\begin{equation}\label{k0 def}
k_0(x,v)=\frac{1}{2\cos(\pi (1/2+v))}\left(J_{2v}(x)-J_{-2v}(x) \right),
\end{equation}

\begin{equation}
k_1(x,v)=\frac{2}{\pi}\sin(\pi(1/2+v))K_{2v}(x).
\end{equation}

\begin{lem}\label{lemma:mellinkernels}
Let \begin{equation}\label{eq:gammauv}
\gamma(u,v)=\frac{2^{2u-1}}{\pi}\Gamma(u+v)\Gamma(u-v).
\end{equation}
For $3/2>Re{w}> 2|\Re{v}|$ we have
\begin{equation}\label{k0 Mellin}
\int_{0}^{\infty}k_0(x,v)x^{w-1}dx=\gamma(w/2,v)\cos{(\pi w/2)},
\end{equation}
and  for $Re{w}> 2|\Re{v}|$
\begin{equation}\label{k1 Mellin}
\int_{0}^{\infty}k_1(x,v)x^{w-1}dx=\gamma(w/2,v)\sin{(\pi (1/2+v))}.
\end{equation}
\end{lem}
\begin{proof}
The equation \eqref{k1 Mellin} follows from \cite[Eq.~6.561.16]{GR}. Let us prove \eqref{k0 Mellin}. Applying \cite[Eq.~6.561.14]{GR} we obtain for $3/2>Re{w}> 2|\Re{v}|$ that
\begin{multline}\label{k0 Mellin2}
\int_{0}^{\infty}k_0(x,v)x^{w-1}dx=
\frac{2^{w-1}}{2\cos(\pi (1/2+v))}\\\times
\left(
\frac{\Gamma(v+w/2)}{\Gamma(1+v-w/2)}-
\frac{\Gamma(-v+w/2)}{\Gamma(1-v-w/2)}
\right).
\end{multline}
Then \cite[Eq. 5.5.3]{HMF} yields
\begin{multline}\label{k0 Mellin3}
\int_{0}^{\infty}k_0(x,v)x^{w-1}dx=
\frac{2^{w-1}}{2\pi\cos(\pi (1/2+v))}\Gamma(v+w/2)\Gamma(w/2-v)
\\\times
\left[
\sin(\pi(w/2-v))-\sin(\pi(w/2+v))
\right].
\end{multline}
Applying \cite[Eq. 4.21.7]{HMF} we obtain \eqref{k0 Mellin}.
\end{proof}

Consider the series
\begin{equation}\label{DirL}
D_{v}(s,x):=\sum_{n\geq 1}\frac{\tau_v(n)}{n^s}e(nx),  \quad \Re{s}>1.
\end{equation}
Let $x$ be a rational number $x=\frac{d}{c}$ with $(d,c)=1$, $c \geq 1$. Then
the function $D_{v}(s,x)$ of two complex parameters $s$ and $v$ is meromorphic on the whole complex plane.  If we fix $v$ such that $\Re{v}=0$ and $v \neq 0$, then $D_{v}(s,d/c)$, as a function of single variable $s$, has two simple poles at $s=1+v$ and $s=1-v$ with the residues
$c^{-1-2v}\zeta(1+2v)$ and $c^{-1+2v}\zeta(1-2v)$, respectively, and it is regular elsewhere. Also it satisfies the functional equation (see \cite[Lemma~3.7]{Mot})
\begin{multline}\label{eq:Dfunct}
D_{v}\left(s,\frac{d}{c}\right)=\left( \frac{4\pi}{c}\right)^{2s-1}\gamma(1-s,v)\times \\ \left\{-\cos{(\pi s)}D_{v}\left(1-s,-\frac{d^*}{c}\right)+\sin{(\pi (1/2+v))}D_{v}\left(1-s,\frac{d^*}{c}\right)\right\},
\end{multline}
where
$dd^*\equiv 1\pmod c$ and $\gamma(u,v)$ is defined by \eqref{eq:gammauv}.
For $\Re{s}<0$ the following estimate is satisfied (see \cite[Eq. 3.3.24]{Mot})
\begin{equation}\label{subD}
|D_{v}(s,d/c)|\ll (c|s|)^{1-2\Re{s}}(\log{|s|})^2.
\end{equation}
Note that
\begin{multline}\label{eq:DZ}
D_{v}(s,d/c)=\sum_{n=1}^{\infty}\frac{\tau_v(n)}{n^s}e\left(n \frac{d}{c} \right)=\frac{1}{c^{2s}}\sum_{a,b=1}^{c}e\left( ab\frac{d}{c}\right)\zeta(a/c;s-v)\zeta(b/c;s+v),
\end{multline}
where
\begin{equation}
\zeta(\alpha;s)=\sum_{n+\alpha>0}\frac{1}{(n+\alpha)^s},\quad \Re{s}>1
\end{equation}
is the Lerch zeta function.
Applying the Euler-Maclaurin formula, we have (see \cite[Lemma~3,~p.~16]{VK})
\begin{equation}\label{eq:EM}
\zeta(\alpha;s)=\sum_{n=0}^{N}\frac{1}{(n+\alpha)^s}+\frac{1}{s-1}(N+1/2+\alpha)^{1-s}+s\int_{N+1/2}^{\infty}\frac{1/2-\{u\}}{(u+\alpha)^{s+1}}du,
\end{equation}
where $\{u\}$ is a fractional part of $u$.
For any $\epsilon>0$ we estimate the absolute value of \eqref{eq:EM}, obtaining
\begin{equation}\label{zetaest}
|\zeta(\alpha;\sigma+iT)|\ll_{\epsilon}
\begin{cases}
1& \sigma\geq 1+\epsilon\\
\log{T}& 1 \leq \sigma \leq 1+\epsilon \\
T^{1-\sigma}& \epsilon \leq \sigma<1
\end{cases}
.
\end{equation}

The  Mellin transform of function $f$ is defined by
\begin{equation}
\hat{f}(z)=\int_{0}^{\infty}t^{z-1}f(t)dt.
\end{equation}
\begin{lem}\label{lemma:change}(Parseval's inequality)
Assume that for some $a\in \R$
\begin{equation}\label{condg}
\int_{0}^{\infty}|g(x)|x^{-a}dx<\infty,\quad \int_{(a)}|\hat{\phi}(z)|dz<\infty.
\end{equation}
Then
\begin{equation}
\frac{1}{2\pi i}\int_{(a)}\hat{\phi}(z)\hat{g}(1-z)dz=\int_{0}^{\infty}\phi(x)g(x)dx.
\end{equation}
\end{lem}
\begin{proof}
See, for example, \cite[Section~3.1.3]{PK}.
\end{proof}
\section{The first moment}
In this section we derive the following asymptotic formula for the first moment of automorphic $L$-functions at the critical point.
\begin{thm}\label{firstmoment} For $2k \geq 12$, $l <k/4 \pi e$ we have
\begin{equation}
\sum_{f \in H_{2k}(1)}^{h}\lambda_f(l)L_f(1/2)=l^{-1/2}(1+i^{2k})+O\left(\frac{1}{\sqrt{l}}\left(2\pi e\frac{l}{k}\right)^k\right).
\end{equation}
\end{thm}
\begin{proof}
This is a consequence of \cite[Theorem 3.1]{BF3} and \cite[Theorem 1.7]{BalFrol}.
\end{proof}


\section{The second moment}

\subsection{Voronoi's summation formula}

Standard Voronoi's summation formulas (see \cite{Jut}) are stated for a function $\phi$ with compact support. However,  weaker conditions on $\phi$  are required to prove a convolution formula for the second moment.

\begin{lem}\label{voronoi}(Kuznetsov, $1981$)
Assume that for $\phi:[0,\infty)\rightarrow \C$, its Mellin transform $\hat{\phi}(s)$ satisfies the following conditions:
\begin{enumerate}
\item
 $\hat{\phi}(2s)$ is regular in the region $\sigma_0<\Re{s}<\sigma_1$ for some $\sigma_1>1$ and $\sigma_0<0$;
\item
for some $\epsilon>0$ and for $\sigma_0<\sigma<\sigma_1$ the function
\begin{equation*}
\biggl((1+|t|)^{1-2\sigma+\epsilon}+1\biggr)|\hat{\phi}(2\sigma+2it)|
\end{equation*}
can be integrated on $(-\infty, +\infty)$.
\end{enumerate}
Then for any $v$ with $\Re{v}=0$, $v\neq0$ and for any coprime integers $c,d \geq 1$ we have
\begin{multline}\label{kuzth1}
\frac{4\pi}{c}\sum_{n = 1}^{\infty}\tau_v(n)e \left(\frac{nd}{c} \right)\phi\left( \frac{4\pi \sqrt{n}}{c}\right)=2\frac{\zeta(1+2v)}{(4\pi)^{1+2v}}\hat{\phi}(2+2v)+
2\frac{\zeta(1-2v)}{(4\pi)^{1-2v}}\hat{\phi}(2-2v)\\+\sum_{n = 1}^{\infty}\tau_v(n) \int_{0}^{\infty}\left(e\left(-\frac{nd^*}{c} \right)k_0(x\sqrt{n},v)+e\left(\frac{nd^*}{c} \right)k_1(x\sqrt{n},v) \right)\phi(x)xdx,
\end{multline}
where $dd^{*}\equiv 1\pmod{c}$.
\end{lem}
Originally, Lemma \ref{voronoi} was proved by Kuznetsov in his doctoral thesis ($1981$) and it was also published in \cite{Kuz}. Unfortunately, the book \cite{Kuz} is hard to find, so we provide all details here.

The proof of Lemma \ref{voronoi} is based on the properties of the series $D_{v}(s,x)$ defined by \eqref{DirL}.
Applying the inverse Mellin transform, we have
\begin{equation*}\phi\left( \frac{4\pi \sqrt{n}}{c}\right)=\frac{1}{i\pi}\int_{(b)}\hat{\phi}(2s)\left(\frac{c}{4\pi}\right)^{2s}\frac{1}{n^s}ds\text{, }1<b<\sigma_1.
\end{equation*}
Therefore,
\begin{equation*}
\frac{4\pi}{c}\sum_{m = 1}^{\infty}e\left(\frac{md}{c}\right)\tau_v(m)\phi \left(\frac{4\pi\sqrt{m}}{c}\right)=\frac{1}{i\pi}\int_{(b)}\left(\frac{c}{4\pi}\right)^{2s-1} D_{v}\left(s,\frac{d}{c}\right)\hat{\phi}(2s)ds.
\end{equation*}
The change of the order of integration and summation in the formula above is allowed since $\int_{(b)}|\hat{\phi}(2s)|ds<\infty$ and the series $D_{v}(s,d/c)$ converges absolutely for $\Re{s}=b>1$.

Moving the contour of integration to $\Re{s}=\delta$ with $\sigma_0<\delta<0$, we cross two simple poles of $D_{v}\left(s,\frac{d}{c}\right)$ at the points $1+v$ and $1-v$. Computation of the residues gives the first two summands on the right-hand side of \eqref{kuzth1}.
To justify this contour shift, we show that for
\begin{equation*}
f(s):=\left(\frac{c}{4\pi}\right)^{2s-1} D_{v}\left(s,\frac{d}{c}\right)\hat{\phi}(2s)
\end{equation*}
one has
\begin{equation}\label{eq:shift1}
\int_{\Re{s}=b, |\Im{s}|>T}f(s)ds\rightarrow 0 \quad \text{as} \quad T \rightarrow \infty,
\end{equation}
\begin{equation}\label{eq:shift2}
\int_{\Re{s}=\delta, |\Im{s}|>T}f(s)ds\rightarrow 0 \quad \text{as} \quad T \rightarrow \infty,
\end{equation}
\begin{equation}\label{eq:shift3}
\int_{\delta}^{b}f(\sigma\pm iT)d\sigma\rightarrow 0 \quad \text{as} \quad T \rightarrow \infty.
\end{equation}
Note that \eqref{eq:shift1} is satisfied since
\begin{equation*}
\int_{T}^{\infty}|\hat{\phi}(2b+2iy)|dy \rightarrow 0 \quad \text{as} \quad T \rightarrow \infty.
\end{equation*}
The property \eqref{eq:shift2} follows from the inequality \eqref{subD} and
\begin{equation*}
\int_{(\delta)}(|s|+1)^{1-2\delta+\epsilon}|\hat{\phi(2s)}|ds<\infty.
\end{equation*}
We split the integral in \eqref{eq:shift3} into two parts $\int_{\delta}^{\epsilon}+\int_{\epsilon}^{b}$.
For the first part we apply the functional equation \eqref{eq:Dfunct} and estimate everything by the absolute value using \eqref{eq:DZ}, \eqref{zetaest}.
The second part is evaluated using the expression \eqref{eq:DZ} and the estimates \eqref{zetaest}. This implies \eqref{eq:shift3}.

Finally, we compute
\begin{equation*}
\frac{1}{i\pi}\int_{(\delta)}\left(\frac{c}{4\pi}\right)^{2s-1}D_{v}\left(s,\frac{d}{c}\right)\hat{\phi}(2s)ds
\end{equation*}
by applying the functional equation \eqref{eq:Dfunct}.  Since $\Re{(1-s)}>1$, we switch the order of summation and integration, obtaining
\begin{multline*}\sum_{m \geq 1}\tau_v(m)\frac{1}{i\pi}\int_{(\delta)}\gamma(1-s,v)\hat{\phi}(2s) m^{s-1}\times \\
\left( -e\left(-\frac{md^{*}}{c}\right)\cos{(\pi s)}+e\left(\frac{md^{*}}{c}\right)\sin{(\pi (1/2+v))}\right)ds.
\end{multline*}
Now the contour of integration can be moved to $\Re{s}=\alpha$ so that $3/4<\alpha<1$.
Then the result follows from Lemmas \ref{lemma:change} and \ref{lemma:mellinkernels}, as we now show.
Let \begin{equation*}
g_1(x):=xk_0(x\sqrt{m},v),
\end{equation*} then
\begin{equation*}
\hat{g_1}(1-2s)=-\gamma(1-s,v)\cos{(\pi s)}m^{s-1}
\end{equation*} and
\begin{equation*}
-\frac{1}{i\pi}\int_{(\alpha)}\gamma(1-s,v)\cos{(\pi s)}\hat{\phi}(2s) m^{s-1}ds=\int_{0}^{\infty}k_0(x \sqrt{m},v) \phi(x)xdx.
\end{equation*}
The parameter $\alpha$ is chosen such that the condition \eqref{condg} is satisfied for $g_1(x)$, i.e.
\begin{equation*}\begin{cases}
   1-2\alpha>-1 & \text{as } x \rightarrow 0, \\
   1/2-2\alpha<-1       & \text{as } x \rightarrow \infty.
  \end{cases}
\end{equation*}

Similarly,
\begin{equation*}\frac{1}{i\pi}\int_{(\alpha)}\gamma(1-s,v)\sin{(\pi (1/2+v))}\hat{\phi}(2s) m^{s-1}ds = \int_{0}^{\infty}k_1(x \sqrt{m},v) \phi(x)xdx.
\end{equation*}
This concludes the proof of Lemma \ref{voronoi}.

\subsection{Convolution formula for the second moment}
Exact formulas for moments reveal the structure of the mean values and allow obtaining asymptotic expansions. Here we use an exact formula for the second twisted moment proved by Kuznetsov. Similar formula was also independently obtained by Iwaniec and Sarnak, see \cite[Theorem 17]{IS}.
\begin{thm}\label{thm:kuznetsov}(Kuznetsov, preprint $1994$)
For $\Re{v}=0$, $\Im{v}\neq 0$, $|\Re{u}|<k-1$ we have
\begin{multline}\label{eq:secondmoment}
M_2(l;u,v):=\sum_{f \in H_{2k}(1)}^{h}\lambda_f(l)L_f(1/2+u+v)L_f(1/2+u-v)=\\
\tau_v(l)\left(\frac{\zeta(1+2u)}{l^{1/2+u}} +\frac{(2\pi)^{4u}}{l^{1/2-u}}\zeta(1-2u)\frac{\Gamma(k-u+v)\Gamma(k-u-v)}{\Gamma(k+u+v)\Gamma(k+u-v)}\right)+\\
\\(-1)^k\tau_u(l) \frac{\zeta(1+2v)}{(2\pi)^{-2u+2v}l^{1/2+v}}\frac{\Gamma(k-u+v)}{\Gamma(k+u-v)}+\\
(-1)^k\tau_u(l)\frac{\zeta(1-2v)}{(2\pi)^{-2u-2v}l^{1/2-v}}\frac{\Gamma(k-u-v)}{\Gamma(k+v+u)}+E(l;u,v).
\end{multline}
The summand $E(l;u,v)$ can be expressed in terms of hypergeometric functions
\begin{multline}\label{error}
E(l;u,v)=\frac{(-1)^k}{\sqrt{l}}\sum_{1 \leq n \leq l-1}\tau_v(n)\tau_u(l-n)\phi_k\left(\frac{n}{l};u,v\right)+\\
\frac{1}{\sqrt{l}}\sum_{ n \geq l+1}\tau_v(n)\tau_u(n-l)\Phi_k\left(\frac{l}{n};u,v\right)+
\frac{(-1)^k}{\sqrt{l}}\sum_{ n \geq 1}\tau_v(n)\tau_u(n+l)\psi_k\left(\frac{l}{n};u,v\right),
\end{multline}
where
\begin{equation}\label{phik}
\phi_k(x;u,v)=\tilde{\phi}_k(x;u,v)+\tilde{\phi}_k(x;u,-v),
\end{equation}
\begin{multline}\label{eq:tildf}
\tilde{\phi}_k(x;u,v)=\frac{(2\pi)^{2u+1}}{2\cos{\left(\pi(1/2+v)\right)}}\frac{\Gamma(k-u+v)}{\Gamma(2v+1)\Gamma(k+u-v)}\times \\x^{v}(1-x)^{-u} F(k-u+v,1-k-u+v,1+2v;x),
\end{multline}
\begin{multline}\label{Phi_k2}
\Phi_k(x;u,v)=2(2\pi)^{2u}\frac{\Gamma(k-u+v)\Gamma(k-u-v)}{\Gamma(2k)}\times\\
\sin{(\pi(1/2+u))}x^k(1-x)^{-u}F(k-u+v,k-u-v,2k;x),
\end{multline}
\begin{multline}\label{psi_k}
\psi_k(x;u,v)=2(2\pi)^{2u}\frac{\Gamma(k-u+v)\Gamma(k-u-v)}{\Gamma(2k)}\times\\
\sin{(\pi(1/2+v))}x^k(1+x)^{-u}F(k-u+v,k-u-v,2k;-x).
\end{multline}

\end{thm}

\begin{proof}
Assume that $k-1>\Re{u}>3/4$. We multiply the both sides of the Petersson trace formula \eqref{Pet} by
\begin{equation*}
n^{-1/2-u}\tau_{v}(n)\zeta(1+2u)
\end{equation*}
and sum over $n \geq 1$. Using the relation \eqref{Lprod},
we obtain the first summand on the right-hand side of \eqref{eq:secondmoment} plus the non-diagonal contribution
\begin{equation*}
M_2(l;u,v)= \zeta(1+2u)\tau_v(l)l^{-1/2-u}+M^{ND},
\end{equation*}
\begin{equation*}
M^{ND}=2\pi i^{2k}\zeta(1+2u)\sum_{n\geq 1}\sum_{c \geq 1}\frac{S(l,n;c)}{c}\frac{\tau_v(n)}{n^{1/2+u}}J_{2k-1}\left(\frac{4\pi\sqrt{ln}}{c} \right).
\end{equation*}
Applying Weil's bound for Kloosterman sums and standard estimates for the $J$-Bessel function, the double series can be bounded as follows
\begin{equation*}
l^{\epsilon}\sum_{n \geq 1}\sum_{c\geq 1} c^{-1/2+\epsilon}n^{-\Re{u}-1/2+\epsilon}\min{\left(\left(\frac{\sqrt{ln}}{c} \right)^{2k-1}, \frac{\sqrt{c}}{(ln)^{1/4}}\right)}\ll
 l^{1/4+\epsilon}\sum_{n\geq 1}n^{-\Re{u}-1/4+\epsilon},
\end{equation*}
where $\epsilon$ is an arbitrary small positive number. Thus for $\Re{u}>3/4$ the series is absolutely convergent and we can change the order of summation in $M^{ND}$. Opening the Kloosterman sum,
we obtain
\begin{multline*}
M^{ND}=2\pi i^{2k}\zeta(1+2u)(4\pi)^{2u}\sum_{c \geq1}\frac{1}{c^{1+2u}}
\sum_{\substack{a\pmod{c}\\ (a,c)=1}}e\left( \frac{al}{c}\right)\\\times
\frac{4\pi}{c}
\sum_{n\geq 1}
\tau_v(n)e\left( \frac{a^*n}{c}\right)\phi\left(\frac{4\pi\sqrt{n}}{c} \right),
\end{multline*}
where
\begin{equation}\label{phi(x)def}
\phi(x)=x^{-1-2u}J_{2k-1}(x\sqrt{l}).
\end{equation}
It follows from \cite[Eq.~6.561.14]{GR} that for $1-k+\Re{u}<\Re{s}<5/4+\Re{u}$
\begin{equation*}
\hat{\phi}(2s)=2^{2s-2-2u}l^{-s+1/2+u}\frac{\Gamma(k+s-1-u)}{\Gamma(k-s+1+u)}.
\end{equation*}
Applying \cite[Eq. 5.11.9]{HMF} to estimate the gamma factors, we obtain
\begin{equation*}
|\hat{\phi}(2\sigma+2it)|\ll(1+|t|)^{-2\sigma-2-2\Re{u}},\quad |t|\rightarrow\infty.
\end{equation*}
Thus  $\phi(x)$ satisfies the conditions of Lemma \ref{voronoi}, which yields
\begin{multline*}
M^{ND}=2 \pi i^{2k}(4\pi)^{2u}\zeta(1+2u)\sum_{c=1}^{\infty}\frac{1}{c^{1+2u}}
\Biggl(2S(0,l;c)\frac{\zeta(1+2v)}{(4\pi)^{1+2v}}\hat{\phi}(2+2v)\\+2S(0,l;c)\frac{\zeta(1-2v)}{(4\pi)^{1-2v}}\hat{\phi}(2-2v)+
\sum_{n=1}^{\infty}\tau_v(n)S(0,l-n;c)\int_{0}^{\infty}k_0(x\sqrt{n},v)\phi(x)xdx\\+
\sum_{n=1}^{\infty}\tau_v(n)S(0,l+n;c)\int_{0}^{\infty}k_1(x\sqrt{n},v)\phi(x)xdx
\Biggr).
\end{multline*}
Ramanujan's identity \eqref{ramid} and \cite[Eq.~6.561(14)]{GR}  allow us to express the first two terms in $M^{ND}$   for $-1/4<\Re{u}<k$ as the third and the fourth summands on the right-hand side of \eqref{eq:secondmoment}.
The second summand in \eqref{eq:secondmoment} comes from the third term in $M^{ND}$ when $n=l$ by applying  \cite[Eq.~6.574(2)]{GR} for $0<\Re{u}<k$.

Consider
\begin{equation}\label{KuzTh.ND1}
2 \pi i^{2k}(4\pi)^{2u}\zeta(1+2u)\sum_{c=1}^{\infty}\frac{1}{c^{1+2u}}\sum_{\substack{n=1\\n \neq l}}^{\infty}\tau_v(n)S(0,l-n;c)
\int_{0}^{\infty}k_0(x\sqrt{n},v)\phi(x)xdx.
\end{equation}
According to \eqref{k0 def} and \eqref{phi(x)def} we have
\begin{multline}\label{k0phi int1}
\int_{0}^{\infty}k_0(x\sqrt{n},v)\phi(x)xdx=
\frac{1}{2\cos(\pi (1/2+v))}\\\times
\left(
\int_{0}^{\infty}
J_{2v}(x\sqrt{n})J_{2k-1}(x\sqrt{l})x^{-2u}dx-
\int_{0}^{\infty}
J_{-2v}(x\sqrt{n})J_{2k-1}(x\sqrt{l})x^{-2u}dx
\right).
\end{multline}
If $n<l$ we apply \cite[Eq.~6.574(1)]{GR} (with $\alpha=\sqrt{n},\,\beta=\sqrt{l},\,\nu=2v,\,\mu=2k-1$),  and obtain for $-1/2<\Re{u}<k$ that
\begin{multline}\label{k0phi int2}
\int_{0}^{\infty}
J_{2v}(x\sqrt{n})J_{2k-1}(x\sqrt{l})x^{-2u}dx=\frac{n^v}{2^{2u}l^{v-u+1/2}}\frac{\Gamma(k+v-u)}{\Gamma(k-v+u)\Gamma(1+2v)}\\\times
F(k+v-u,1-k+v-u,1+2v;n/l)=\frac{2\cos(\pi (1/2+v))\tilde{\phi}_k(n/l;u,v)}{2^{2u}(2\pi)^{2u+1}(1-n/l)^{-u}l^{1/2-u}},
\end{multline}
 where we used \eqref{eq:tildf}. Substituting \eqref{k0phi int2} to \eqref{k0phi int1} and applying \eqref{phik}, we show that
\begin{equation}\label{k0phi int3}
\int_{0}^{\infty}k_0(x\sqrt{n},v)\phi(x)xdx=
\frac{\phi_k(n/l;u,v)}{2^{2u}(2\pi)^{2u+1}(1-n/l)^{-u}l^{1/2-u}}.
\end{equation}
Recall that we consider the part of \eqref{KuzTh.ND1} with $n<l$.  Substituting \eqref{k0phi int3} to \eqref{KuzTh.ND1}
and using Ramanujan's identity \eqref{ramid} to compute the sum over $c$, we recover the first term in
\eqref{error}.
\par
Analogously, for $n>l$ we obtain the second term in \eqref{error} using \cite[Eq.~6.574(1)]{GR} with $\alpha=\sqrt{l},\,\beta=\sqrt{n},\,\nu=2k-1,\,\mu=2v$.
\par
Finally, the third term in \eqref{error} comes from
\begin{equation*}
2 \pi i^{2k}(4\pi)^{2u}\zeta(1+2u)\sum_{c=1}^{\infty}\frac{1}{c^{1+2u}}\sum_{n=1}^{\infty}\tau_v(n)S(0,l+n;c)
\int_{0}^{\infty}k_1(x\sqrt{n},v)\phi(x)xdx
\end{equation*}
by applying  \eqref{ramid} and \cite[Eq.~6.576(3)]{GR} for $\Re{u}<k$.
\par
Thus we proved \eqref{eq:secondmoment} for $k-1>\Re{u}>3/4$. To extend the range of validity of \eqref{eq:secondmoment}, we note that the left-hand side of the convolution formula \eqref{eq:secondmoment} is the entire function of $u$ and $v$. Since
\begin{equation*}|\tau_{u}(n\pm l)|\ll n^{|\Re{u}|+\epsilon}, \quad |\Phi_k(x;u,v)|,|\psi_k(x;u,v)|\ll x^k\text{ as }x \rightarrow 0,
\end{equation*}
the right-hand side of \eqref{eq:secondmoment} is regular function for $|\Re{v}|+|\Re{u}|<k-1$.

\end{proof}




\section{The Liouville-Green method}\label{section:LG}

Our main references are the paper \cite{BD} and the book  \cite{O} .
In particular, Chapters $6$, $10-12$ of the book \cite{O} are devoted to the Liouville-Green method. Note that the case we are interested in is mainly covered by Chapter $12$. 

Throughout this section we assume that $k$ is an even positive integer.

\subsection{Some properties of $\phi_k$} Let $0<x<1$ be a real number. Consider the function
\begin{equation*}
\phi_k(x)=\lim_{\substack{u \rightarrow 0\\ v \rightarrow 0}}\phi_k(x;u,v),
\end{equation*}
where $\phi_k(x;u,v)$ is defined by the equation \eqref{phik}.
Letting $u=0$ and computing the limit as $v\rightarrow 0$ by L'Hospital's rule, we obtain
\begin{equation}\label{phidef}
\phi_k(x)=\frac{\partial}{\partial v}\Biggl[\frac{-2\Gamma(k+v)x^v}{\Gamma(1+2v)\Gamma(k-v)}{F(k+v,1-k+v,1+2v;x)} \Biggr]\Bigg|_{v=0}.
\end{equation}
Differentiation with respect to $v$ gives
\begin{multline}\label{phiu}
\phi_k(x)=2\left( -\log{x}-2\frac{\Gamma '}{\Gamma}(k)+2\frac{\Gamma '}{\Gamma}(1)\right)F(k,1-k,1;x)-\\
2\left( \frac{\partial}{\partial a}+\frac{\partial}{\partial b}+2\frac{\partial}{\partial c} \right)F(a,b,c;x)\Bigg|_{\substack{a=k\\b=1-k\\c=1}}.
\end{multline}

\begin{lem} The following property holds
\begin{equation}\label{funceqphi}
\phi_k(x)=(-1)^k\phi_k(1-x).
\end{equation}
\end{lem}
\begin{proof}
Recall that
\begin{equation*}
\phi_k(x;u,v)=\tilde{\phi}_k(x;u,v)+\tilde{\phi}_k(x;u,-v),
\end{equation*}
where $\tilde{\phi}_k(x;u,v)$ is defied by equation \eqref{eq:tildf}.
Applying \cite[Eq.~33,~p.~107]{BE} and Euler's reflection formula, we obtain
\begin{multline*}
\tilde{\phi}_k(x;u,v)=(-1)^k\frac{(2\pi)^{2u}\pi}{\sin{(\pi v)}}\biggl(
\frac{\Gamma(k+v-u)\Gamma(k-v-u)}{\Gamma(k+v+u)\Gamma(k-v+u)}\times\\
\frac{\sin{(\pi(v+u))}}{\Gamma(1-2u)\sin{(2\pi u)}}x^{v}(1-x)^{-u}F(k+v-u,1-k+v-u,1-2u;1-x)+\\
\frac{\sin{(\pi(v-u))}}{\Gamma(1+2u)\sin{(-2\pi u)}}x^{v}(1-x)^{u}F(k+v+u,1-k+v+u,1+2u;1-x)
\biggr).
\end{multline*}
Let
\begin{equation*}
F(v,u):=x^v(1-x)^{u}F(k+v+u,1-k+v+u,1+2u;1-x).
\end{equation*}
Then
\begin{multline*}
\tilde{\phi}_k(x;u,v)=(-1)^k\frac{(2\pi)^{2u}\pi}{2\sin{(\pi u)}}\biggl(
\frac{\Gamma(k+v-u)\Gamma(k-v-u)}{\Gamma(k+v+u)\Gamma(k-v+u)}\frac{F(v,-u)}{\Gamma(1-u)}-
\frac{F(v,u)}{\Gamma(1+2u)}\biggr)\\+(-1)^k\frac{(2\pi)^{2u}\pi}{2\sin{(\pi v)}}\frac{\cos{(\pi v)}}{\cos{(\pi u)}}
\biggl( \frac{\Gamma(k+v-u)\Gamma(k-v-u)}{\Gamma(k+v+u)\Gamma(k-v+u)}\frac{F(v,-u)}{\Gamma(1-2u)}+
\frac{F(v,u)}{\Gamma(1+2u)}\biggr).
\end{multline*}
Computing the limit as $u \rightarrow 0$ by L'Hospital rule, we have
\begin{multline*}
\lim_{u \rightarrow 0}\tilde{\phi}_k(x;u,v)=(-1)^k\frac{\pi \cos{(\pi v)}}{\sin{ (\pi v)}}F(v,0)\\+\frac{(-1)^k}{2} 
 \frac{\partial}{\partial u}\Biggl(
\frac{\Gamma(k+v-u)\Gamma(k-v-u)}{\Gamma(k+v+u)\Gamma(k-v+u)}\frac{F(v,-u)}{\Gamma(1-2u)}
-\frac{F(v,u)}{\Gamma(1+2u)} \Biggr)\Bigg|_{u=0}.
\end{multline*}
By \cite[Eq.~1-2,~p.~105]{BE} it follows that
$F(-v,0)=F(v,0).$
Thus
\begin{multline*}
\lim_{u \rightarrow 0}\left(\tilde{\phi}_k(x;u,v)+\tilde{\phi}_k(x;u,-v)\right)=\\\frac{(-1)^k}{2} \Biggl[ \frac{\partial}{\partial u}\Biggl(
\frac{\Gamma(k+v-u)\Gamma(k-v-u)}{\Gamma(k+v+u)\Gamma(k-v+u)}\frac{F(v,-u)}{\Gamma(1-2u)}
-\frac{F(v,u)}{\Gamma(1+2u)} \Biggr)\Bigg|_{u=0}+\\
 \frac{\partial}{\partial u}\Biggl(
\frac{\Gamma(k-v-u)\Gamma(k+v-u)}{\Gamma(k-v+u)\Gamma(k+v+u)}\frac{F(-v,-u)}{\Gamma(1-2u)}
-\frac{F(-v,u)}{\Gamma(1+2u)} \Biggr)\Bigg|_{u=0}\Biggr].
\end{multline*}
Letting $v=0$, we have
\begin{equation*}
\phi_k(x)=(-1)^k4\left(\frac{\Gamma'}{\Gamma}(1)-\frac{\Gamma'}{\Gamma}(k)\right)F(0,0)-(-1)^k2F^{'}_{u}(0,0),
\end{equation*}
where
\begin{equation*}
F^{'}_{u}(0,0)=\log{(1-x)}F(k,1-k,1;1-x)+
\left( \frac{\partial}{\partial a}+\frac{\partial}{\partial b}+2\frac{\partial}{\partial c} \right)F(a,b,c;1-x)\Bigg|_{\substack{a=k\\b=1-k\\c=1}}.
\end{equation*}
Then the formula \eqref{phiu} implies that $\phi_k(x)=(-1)^k\phi_k(1-x)$.
\end{proof}


\begin{cor} \label{phi derivative 1/2}
For any positive even integer $k$ we have $\phi_k'(1/2)=0$.
\end{cor}


\begin{lem}
The following series representation holds
\begin{multline}\label{series}
\phi_k(x)=- F(k,1-k,1;x)2\log{x}+
2(-1)^k\sum_{n=k}^{\infty}\frac{\Gamma(n+k)\Gamma(n-k+1)x^n}{\Gamma^2(n+1)}\\-
 2\sum_{n=0}^{k-1}\frac{(-1)^n\Gamma(k+n)}{\Gamma(k-n)\Gamma^2(n+1)}
\left(-2\frac{\Gamma '}{\Gamma}(n+1)+\frac{\Gamma '}{\Gamma}(k+n)+\frac{\Gamma '}{\Gamma}(k-n)  \right).
\end{multline}
\end{lem}
\begin{proof}
By \eqref{phidef} and Euler's reflection formula
\begin{equation*}
\phi_k(x)=-\frac{2}{\pi}\frac{\partial}{\partial v}\Biggl(x^v\sum_{n=0}^{\infty}\frac{\sin{(\pi(k-v))}\Gamma(1-k+v+n)\Gamma(k+v+n)}{\Gamma(1+2v+n)}\frac{x^n}{n!} \Biggr)\Bigg|_{v=0}.
\end{equation*}
Using
\begin{equation*}
\Gamma(1-k+v+n)=\begin{cases}
\Gamma(1-k+v+n) & n \geq k\\
\frac{\pi}{\sin{(\pi(k-v-n))}\Gamma(k-n-v)}& n \leq k-1
\end{cases},
\end{equation*}
we have
\begin{multline}\label{phik series}
\phi_k(x)=2(-1)^k\sum_{n=k}^{\infty}\frac{\Gamma(n+k)\Gamma(n-k+1)x^n}{\Gamma(n+1)n!}+\sum_{n=0}^{k-1}\frac{(-1)^nx^n}{n!}\frac{\Gamma(k+n)}{\Gamma(k-n)\Gamma(n+1)}\\\times
\left(-2\log{x}+4\frac{\Gamma '}{\Gamma}(n+1)-2\frac{\Gamma '}{\Gamma}(k+n)-2\frac{\Gamma '}{\Gamma}(k-n)  \right).
\end{multline}
It follows from \cite[Eq.~15.2.4]{HMF} that
\begin{equation*}
F(-m,b,c;z)=\frac{\Gamma(m+1)\Gamma(c)}{\Gamma(b)}\sum_{n=0}^{m}\frac{(-1)^nz^n}{n!}\frac{\Gamma(b+n)}{\Gamma(m-n+1)\Gamma(c+n)},
\end{equation*}
and therefore,
\begin{equation}\label{hypergeom negative int}
F(1-k,k,1;x)=\sum_{n=0}^{k-1}\frac{(-1)^nx^n}{n!}\frac{\Gamma(k+n)}{\Gamma(k-n)\Gamma(n+1)}.
\end{equation}
Substituting \eqref{hypergeom negative int} to \eqref{phik series} we obtain \eqref{series}.
\end{proof}


\begin{lem}
The function $\phi_k(x)$ satisfies the differential equation
\begin{equation}\label{diffur1}
(x-x^2)\phi_{k}^{''}(x)+(1-2x)\phi_{k}^{'}+k(k-1)\phi_k(x)=0.
\end{equation}
\end{lem}
\begin{proof}
Note that $F(k,1-k,1;x)$ is a solution of the equation \eqref{diffur1}.
Using \eqref{series} we can write
\begin{equation*}
\phi_k(x)=-2\alpha_1-2\alpha_2+2(-1)^k\alpha_3,
\end{equation*}
where
\begin{equation*}
\alpha_1:=F(k,1-k,1;x)\log{x},
\end{equation*}
\begin{equation*}
\alpha_2:=\sum_{n=0}^{k-1}A(n)B(n)x^n,\quad \alpha_3:=\sum_{n=k}^{\infty}C(n)x^n.
\end{equation*}
The coefficients $A(n)$, $B(n)$, $C(n)$ are defined by
\begin{equation*}
A(n):=(-1)^n\frac{\Gamma(k+n)}{\Gamma^2(n+1)\Gamma(k-n)},
\end{equation*}
\begin{equation*}
B(n):=\frac{\Gamma'}{\Gamma}(n+k)+\frac{\Gamma'}{\Gamma}(n-k)-2\frac{\Gamma'}{\Gamma}(n+1),
\end{equation*}
\begin{equation*}
C(n):=\frac{\Gamma(n+k)\Gamma(n-k+1)}{\Gamma^2(n+1)}.
\end{equation*}

They satisfy the recurrence relations:
\begin{equation*}
A(n+1)=-\frac{(k+n)(k-n-1)}{(n+1)^2}A(n),
\end{equation*}
\begin{equation*}
B(n+1)=B(n)+\frac{1}{k+n}-\frac{1}{k-n-1}-\frac{2}{n+1},
\end{equation*}
\begin{equation*}
C(n+1)=\frac{(n+k)(n-k+1)}{(n+1)^2}C(n).
\end{equation*}

Let us denote
\begin{equation*}
D(f):=(x-x^2)f''+(1-2x)f'+k(k-1)f.
\end{equation*}
Using the recurrence relations above, we compute $D(\alpha_1)$, $D(\alpha_2)$, $D(\alpha_3)$ and prove the lemma by showing that
\begin{equation*}
D(\alpha_1)+D(\alpha_2)=(-1)^kD(\alpha_3).
\end{equation*}
\end{proof}

\begin{lem}\label{diffurabc}
Let $y=y(x)$ be a solution of the differential equation
\begin{equation}
A(x)y''(x)+B(x)y'(x)+C(x)y(x)=0.
\end{equation}
Then  $z(x)=y(x)/\alpha(x)$ satisfies the equation
\begin{equation}
A_1(x)z''(x)+B_1(x)z'(x)+C_1(x)z(x)=0,
\end{equation}
where
\begin{equation}
A_1(x)=A(x)\alpha(x), \quad B_1(x)=2A(x)\alpha'(x)+B(x)\alpha(x),
\end{equation}
\begin{equation} C_1(x)=A(x)\alpha''(x)+B(x)\alpha'(x)+C(x)\alpha(x).
\end{equation}
\end{lem}

\begin{cor}\label{cor:diffuryk}
For $0<x<1$ the function $Y(x):=\sqrt{x(1-x)}\phi_k(x) $ is a solution of the differential equation
\begin{equation}\label{diffuryk}
Y''(x)+\left(\frac{1}{4x^2(1-x)^2}+\frac{k(k-1)}{x(1-x)} \right)Y(x)=0.
\end{equation}
\end{cor}
\begin{proof}
Applying Lemma \ref{diffurabc} with $\alpha(x)=1/\sqrt{x(1-x)}$, we prove the corollary.
\end{proof}

\begin{lem}\label{lem:derivF}
Assume that $k$ is an even positive integer. Then
\begin{equation}\label{hyp1}
F(k,1-k,1;1/2)=0,
\end{equation}
\begin{equation}\label{hyp2}
\frac{d}{dx}\left(F(k,1-k,1;x)\right)\bigg|_{x=1/2}=(-1)^{k/2}\frac{4\Gamma(1/2)\Gamma((k+1)/2)}{\pi \Gamma(k/2)}.
\end{equation}
\end{lem}
\begin{proof}
Using the formula \cite[Eq.~15.8.25]{HMF} and Euler's reflection formula, we have
\begin{multline*}
F(a,1-a,1;x)=\frac{\Gamma(1/2)}{\pi}\frac{\Gamma(a/2)\sin(\pi a/2)}{\Gamma((a+1)/2)}
F\left(\frac{a}{2},\frac{1-a}{2},\frac{1}{2};(1-2x)^2\right)\\-
(1-2x)\frac{2\Gamma(1/2)}{\pi} \frac{\Gamma((a+1)/2)\sin(\pi (a+1)/2)}{\Gamma(a/2)} F\left(\frac{a+1}{2},1-\frac{a}{2},\frac{3}{2};(1-2x)^2\right)
\end{multline*}
for some complex variable $a$.
Setting $a=k$ we obtain
\begin{multline*}
F(k,1-k,1;x)=-(-1)^{k/2}\frac{2\Gamma(1/2)\Gamma((k+1)/2)}{\pi\Gamma(k/2)}(1-2x)\times \\
F\left(\frac{k+1}{2},1-\frac{k}{2},\frac{3}{2};(1-2x)^2\right).
\end{multline*}
The equation \eqref{hyp1} follows by taking $x=1/2$.

The equation \eqref{hyp2} is obtained by differentiation of $F(k,1-k,1;x)$ with respect to $x$ using the series representation
\begin{multline*}
F\left(\frac{k+1}{2},1-\frac{k}{2},\frac{3}{2};(1-2x)^2\right)=\\
\sum_{n=0}^{k/2-1}(-1)^n\binom{k/2-1}{n}\frac{\Gamma(k/2+1/2+n)\Gamma(3/2)}{\Gamma(k/2+1/2)\Gamma(3/2+n)}(1-2x)^{2n}.
\end{multline*}
This representation is a consequence of  \cite[Eq.~15.2.4]{HMF} for even positive integer $k$.
\end{proof}

\begin{lem}\label{lem:phik12}For even positive integer $k$ we have
\begin{equation}\label{phi12}
\phi_k(1/2)=2\sqrt{\pi}(-1)^{k/2}\frac{\Gamma(k/2)}{\Gamma((k+1)/2)}.
\end{equation}
\end{lem}

\begin{proof}
By \eqref{phiu} and \eqref{hyp1} we have
\begin{equation*}
\phi_k(1/2)=-2\left( \frac{\partial}{\partial a} +\frac{\partial}{\partial b}+2\frac{\partial}{\partial c}\right)F(a,b,c;1/2)\bigg|_{\substack{a=k\\b=1-k\\c=1}}.
\end{equation*}
Define two functions
\begin{equation*}
G_1:=F(k+2\epsilon,1-k,1+\epsilon;z),
\end{equation*}
\begin{equation*}
G_2:=F(k-\epsilon,1-k+\epsilon,1+\epsilon;z).
\end{equation*}
According to \cite[Eq.~15.8.26]{HMF} we have
\begin{multline*}
G_1=\frac{\Gamma(1/2)}{\pi}(1-z)^{\epsilon}\frac{\Gamma(1+\epsilon)\Gamma(k/2-\epsilon)}{\Gamma(k/2+1/2)}\sin{\left(\frac{\pi}{2}(k-2\epsilon)\right)}g_{1,1}+\\
\frac{\Gamma(-1/2)}{\pi}(1-2z)(1-z)^{\epsilon}\frac{\Gamma(1+\epsilon)\Gamma(k/2+1/2-\epsilon)}{\Gamma(k/2)}\sin{\left(\frac{\pi}{2}(k+1-2\epsilon)\right)}g_{1,2},
\end{multline*}
where
\begin{equation*}
g_{1,1}=F(1/2-k/2+\epsilon,k/2,1/2;(1-2z)^2),
\end{equation*}
\begin{equation*}
g_{1,2}=F(1-k/2+\epsilon,k/2+1/2,3/2;(1-2z)^2).
\end{equation*}
By \cite[Eq.~15.8.25]{HMF}
\begin{equation*}
G_2=-\frac{2(-1)^{k/2}\Gamma(1/2)}{\pi}(1-2z)\frac{\Gamma(1+\epsilon)\Gamma(k/2+1/2)}{\Gamma(k/2+\epsilon)}g_{2,2},
\end{equation*}
where
\begin{equation*}
g_{2,2}=F(k/2+1/2+\epsilon,1-k/2,3/2;(1-2z)^2).
\end{equation*}
Differentiating $G_1$ and $G_2$ with respect to $\epsilon$ at the point $\epsilon=0$ and summing the results, we have
\begin{multline*}
\left( \frac{\partial}{\partial a} +\frac{\partial}{\partial b}+2\frac{\partial}{\partial c}\right)F(a,b,c;z)\bigg|_{\substack{a=k\\b=1-k\\c=1}}=-(1-2z)\log{(1-z)} \times\\
(-1)^{k/2}F(k/2,1/2-k/2,1/2;(1-2z)^2)\frac{2\Gamma(k/2+1/2)\Gamma(1/2)}{\pi \Gamma(k/2)}-\\
(-1)^{k/2}F(k/2,1/2-k/2,1/2;(1-2z)^2)\frac{\Gamma(1/2)\Gamma(k/2)}{\Gamma(k/2+1/2)}-\\
\frac{\Gamma(k/2+1/2)\Gamma(1/2)}{\pi\Gamma(k/2)}(1-2z)\Biggl(F(k/2+1/2,1-k/2,3/2;(1-2z)^2)\times \\
(-1)^{k/2}\biggl( 4 \frac{\Gamma'}{\Gamma}(1)-2\frac{\Gamma'}{\Gamma}(k/2+1/2) -2\frac{\Gamma'}{\Gamma}(k/2)\biggr)+\\
2(-1)^{k/2}\biggl( \frac{\partial}{\partial a}+\frac{\partial}{\partial b}\biggr)F(a,b,3/2;(1-2z)^2)\bigg|_{\substack{a=(k+1)/2\\b=1-k/2}}
\Biggr).
\end{multline*}
Setting $z=1/2$ we prove the lemma.
\end{proof}

\subsection{Asymptotic approximation of $\phi_k$}
We apply the Liouville-Green method to find a uniform approximation of the function $\phi_k(x)$ in terms of the $J$ and $Y$ Bessel functions. Some preliminary work is required in order to introduce various functions and constants in the Liouville-Green approximation. 
Therefore, the main result, namely Theorem \ref{thm:approxphi}, is stated at the end of this subsection.

In Corollary \ref{cor:diffuryk} we showed that  $y(x)=\sqrt{x(1-x)}\phi_k(x)$
is a solution of the differential equation \eqref{diffuryk}. This equation is a particular type of \cite[Eq.~1.1]{BD} when $\alpha=0$.
Let
\begin{equation}\label{oldf}
u:=k-1/2,\quad f(x):=-\frac{1}{x(1-x)},
\end{equation}
\begin{equation}
g(x):=-\frac{1}{4x^2(1-x)^2}+\frac{1}{4x(1-x)}.
\end{equation}
Then the equation  \eqref{diffuryk} can be written as
\begin{equation}\label{diffurufg}
y''(x)=(u^2f(x)+g(x))y(x).
\end{equation}
Note that $x^2f(x)\rightarrow 0$ and $x^2g(x)\rightarrow -1/4$ as $x \rightarrow 0$. The same conditions on $f(x)$ and $g(x)$ are also assumed on \cite[page 1]{BD}.
We would like to transform the equation  \eqref{diffurufg} into the following shape
\begin{equation}\label{diffurzxi}
\frac{d^2Z}{d \xi^2}+\left[\frac{u^2}{4\xi}+\frac{1}{4\xi^2}-\frac{\psi(\xi)}{\xi} \right]Z=0,
\end{equation}
which corresponds to \cite[Eq.~2.4,~2.6]{BD}.

Let $\alpha(x)$, $\eta(x)$ be some suitable functions (to be chosen later).
We make the  change of variable
\begin{equation}
Z(x):=\frac{y(x)}{\alpha(x)}
\end{equation}
in the equation \eqref{diffurufg} and apply Lemma \ref{diffurabc}.
Then the substitution
\begin{equation}
\xi:=\int \eta(x)dx
\end{equation}
gives
\begin{multline}\label{diffuretaalpha}
\alpha(x)\eta^2(x)\frac{d^2Z}{d\xi^2}+\left(\alpha(x)\eta'(x)+2\alpha'(x)\eta(x)\right)\frac{dZ}{d\xi}+\\
\left(\alpha''(x)-\alpha(x)(u^2f(x)+g(x))\right)Z(\xi)=0.
\end{multline}
In order to obtain the equation \eqref{diffurzxi} we make the coefficient before $\frac{dZ}{d\xi}$ vanish by
requiring
\begin{equation}\label{req1}
\alpha^2(x)\eta(x)=1.
\end{equation}
Next, we assume that $-\alpha^4(x)f(x)=1/(4 \xi)$. This implies
\begin{equation}\label{req2}
\xi=4\arcsin^2{\left(\sqrt{x}\right)}, \quad \alpha(x)=\frac{(x-x^2)^{1/4}}{2(\arcsin{(\sqrt{x})})^{1/2}}.
\end{equation}

Using \eqref{req1} and \eqref{req2},  the equation \eqref{diffuretaalpha} can be transformed into \eqref{diffurzxi} with
\begin{equation}
\psi(\xi):=-\frac{1}{16 \sin^2{\left(\sqrt{\xi}\right)}}+\frac{1}{16\xi}.
\end{equation}
Note that $\psi(\xi)$ is regular in the $\xi$-plane apart from the poles at $\xi=(\pi m)^2$, $m \neq 0$.
Since
\begin{equation*}
\psi(\xi)=-\frac{1}{48}+O(\xi)\text{ as } \xi \rightarrow 0,
\end{equation*}
the function $\psi(\xi)$ is smooth on the interval $[0,\delta]$ for any $0<\delta<\pi^2$.

Removing the summand with $\psi(\xi)/\xi$ in the equation \eqref{diffurzxi}, we have
\begin{equation}\label{diffurjy}
\frac{d^2Z}{d \xi^2}+\left[\frac{u^2}{4\xi}+\frac{1}{4\xi^2} \right]Z=0.
\end{equation}
Solutions of \eqref{diffurjy} are defined by
\begin{equation}
Z_C=\sqrt{\xi}\Upsilon_0(u\sqrt{\xi}),
\end{equation}
where $\Upsilon_i$ is either the $J$ or $Y$ Bessel function of index $i$. Note that
\begin{equation}
\frac{d}{dz}\Upsilon_0(z)=-\Upsilon_1(z).
\end{equation}
Therefore, following \cite[Chapter~12]{O} we are searching for a solution of the differential equation \eqref{diffurzxi} in the form
\begin{equation}\label{solution}
Z_C(\xi)=\sqrt{\xi}\Upsilon_0(u\sqrt{\xi})\sum_{n=0}^{\infty}\frac{A(n;\xi)}{u^{2n}}-\frac{\xi}{u}\Upsilon_1(u\sqrt{\xi})\sum_{n=0}^{\infty}\frac{B(n;\xi)}{u^{2n}}.
\end{equation}
And our problem reduces to finding the coefficients $A(n;\xi)$, $B(n;\xi)$.
Let us denote
\begin{equation}
W(\xi):=\sqrt{\xi}\Upsilon_0(u\sqrt{\xi}), \quad V(\xi):=\xi \Upsilon_1(u\sqrt{\xi}).
\end{equation}
These functions satisfy the differential equations (see \cite[Eq.~8.491(3)]{GR})
\begin{align}
W''(\xi)+\left( \frac{u^2}{4\xi}+\frac{1}{4\xi^2}\right)W(x)=0,\\
V''(\xi)-\frac{1}{\xi}V'(\xi)+\left( \frac{u^2}{4\xi}+\frac{3}{4\xi^2}\right)V(x)=0.
\end{align}
Note that
\begin{equation}\label{eq:derivW}
W'(\xi)=\frac{1}{2\xi}W(\xi)-\frac{u}{2\xi}V(\xi),
\end{equation}
\begin{equation}\label{eq:derivV}
V'(\xi)=\frac{1}{2\xi}V(\xi)+\frac{u}{2}W(\xi).
\end{equation}
Then, substituting \eqref{solution} in the equation \eqref{diffurzxi}, we find
\begin{equation}
W(\xi)\sum_{n=0}^{\infty}\frac{C_n(\xi)}{u^{2n}}-V(\xi)\sum_{n=0}^{\infty}\frac{D_n(\xi)}{u^{2n-1}}=0,
\end{equation}
where
\begin{equation*}
C_n(\xi):=A''(n;\xi)+\frac{1}{\xi}A'(n;\xi)-\frac{\psi(\xi)}{\xi}A(n;\xi)-B'(n;\xi)-\frac{B(n;\xi)}{2\xi},
\end{equation*}
\begin{equation*}
D_n(\xi):=B''(n-1;\xi)+\frac{1}{\xi}B'(n-1;\xi)-\frac{\psi(\xi)}{\xi}B(n-1;\xi)+\frac{1}{\xi}A'(n;\xi).
\end{equation*}

Assuming that $C_n(\xi)=D_n(\xi)=0$, we have
\begin{equation*}
\sqrt{\xi}(\sqrt{\xi} B(n;\xi))'=\xi A''(n;\xi)+A'(n;\xi)-\psi(\xi)A(n;\xi),
\end{equation*}
\begin{equation*}
A'(n;\xi)=-(\xi B'(n-1;\xi))'+\psi(\xi)B(n-1;\xi).
\end{equation*}

This yields the following recurrence relations
\begin{equation}\label{recurrence1}
A(n;\xi)=-\xi B'(n-1;\xi)+\int_{0}^{\xi}\psi(x)B(n-1;x)dx+\lambda_n,
\end{equation}
\begin{equation}\label{recurrence2}
\sqrt{\xi}B(n;\xi)=\int_{0}^{\xi}\frac{1}{\sqrt{x}}\left(xA''(n;x)+A'(n;x)-\psi(x)A(n;x) \right)dx
\end{equation}
for some real constants of integration $\lambda_n$  (to be chosen later).
Letting $A(0;\xi)=1$, we have
\begin{equation}\label{eq:bo}
B(0;\xi)=-\frac{1}{8\sqrt{\xi}}\left(\cot{\sqrt{\xi}}-\frac{1}{\sqrt{\xi}}\right),
\end{equation}
\begin{equation}\label{eq:a1}
A(1;\xi)=\frac{1}{8}\left(\frac{1}{\xi}-\frac{\cot{\sqrt{\xi}}}{2\sqrt{\xi}}-\frac{1}{2\sin^2{\left(\sqrt{\xi}\right)}} \right)
-\frac{1}{128}\left(\cot{\left(\sqrt{\xi}\right)}-\frac{1}{\sqrt{\xi}} \right)^2+\lambda_1.
\end{equation}

The next step is to apply \cite[Theorem~4.1,~p.~444]{O} or   \cite[Theorem~1]{BD}. This allows approximating $\phi_k$  by a finite series plus an error term.

\begin{thm}\label{LGphi}
Let $\xi_2=\pi^2/4$. For each value of $u$ and each nonnegative integer $N$, the equation \eqref{diffurzxi} has solutions $Z_Y(\xi)$, $Z_J(\xi)$ which are infinitely differentiable in $\xi$ on the interval $(0, \xi_2)$, and are given by
\begin{multline}\label{zyxi}
Z_Y(\xi)=\sqrt{\xi}Y_0(u\sqrt{\xi})\sum_{n=0}^{N}\frac{A_Y(n;\xi)}{u^{2n}}-
\frac{\xi}{u}Y_1(u\sqrt{\xi})\sum_{n=0}^{N-1}\frac{B_Y(n;\xi)}{u^{2n}}+\epsilon_{2N+1,1}(u,\xi),
\end{multline}
\begin{multline}\label{eq:zjapprox}
Z_J(\xi)=\sqrt{\xi}J_0(u\sqrt{\xi})\sum_{n=0}^{N}\frac{A_J(n;\xi)}{u^{2n}}-
\frac{\xi}{u}J_1(u\sqrt{\xi})\sum_{n=0}^{N-1}\frac{B_J(n;\xi)}{u^{2n}}+\epsilon_{2N+1,2}(u,\xi),
\end{multline}
where
\begin{equation}\label{zyxi2}
\epsilon_{2N+1,1}(u,\xi)\ll \frac{\sqrt{\xi}|Y_0(u\sqrt{\xi})|}{u^{2N+1}}\sqrt{\xi_2-\xi},
\end{equation}
\begin{equation}
\epsilon_{2N+1,2}(u,\xi)\ll \frac{\sqrt{\xi}|J_0(u\sqrt{\xi})|}{u^{2N+1}}\min{(\sqrt{\xi},1)}
\end{equation}
and the coefficients $(A_Y(n;\xi),B_Y(n;\xi))$, $(A_J(n;\xi),B_J(n;\xi))$ satisfy the recurrence relations \eqref{recurrence1}-\eqref{recurrence2}.
\end{thm}
\begin{proof}
The only difference between Theorem \ref{LGphi} and \cite[Theorem~4.1,~p.~444]{O} (or   \cite[Theorem~1]{BD}) is that we simplify the estimates of the  error terms $\epsilon_{2N+1,i}(u,\xi)$.  To this end, we apply  \cite[(1.22)-(1.24),~p.~437]{O}. Furthermore, we use the fact (see \eqref{eq:bo}) that 
\begin{equation*}
\Var_{0,\xi}(\sqrt{x}B(0;x))\ll1,\quad \Var_{\xi,\xi_2}(\sqrt{x}B(0;x))\ll1,
\end{equation*}
 and the following estimates for variations
\begin{equation*}
\Var_{0,\xi}(\sqrt{x}B(n;x))=\int_0^{\xi}|(\sqrt{x}B(n;x))'|dx\ll \min{(\sqrt{\xi},1)},
\end{equation*}
\begin{equation*}
\Var_{\xi,\xi_2}(\sqrt{x}B(n;x))=\int_{\xi}^{\xi_2}|(\sqrt{x}B(n;x))'|dx\ll \sqrt{\xi_2-\xi},
\end{equation*}
which can be proved by induction.
\end{proof}

Note that in Theorem \ref{LGphi} it is possible to choose $\xi_2$  to be any number less than $\pi^2$ since the function $\psi(\xi)$ is smooth on the interval $[0,\delta]$ for any $0<\delta<\pi^2$. However, we  let $\xi_2=\pi^2/4$ because this point corresponds to $x=1/2$.  Consequently, Theorem \ref{LGphi} allows to approximate the function $\phi_k(x)$ for $0<x<1/2$, which is sufficient for our purposes in view of the functional equation \eqref{funceqphi}.

\begin{thm}\label{LGphi2}
There are $C_Y=C_Y(u)$ and $C_J=C_J(u)$ such that
\begin{equation}\label{eq:cycj2}
\xi^{1/4}(\sin{(\sqrt{\xi})})^{1/2}\phi_k\left(\sin^2{\left(\frac{\sqrt{\xi}}{2}\right)} \right)=C_YZ_Y(\xi)+C_JZ_J(\xi).
\end{equation}
\end{thm}

Now our goal is not only to prove Theorem \ref{LGphi2} but also to determine explicitly $C_Y$ and $C_J$.

As $\xi \rightarrow 0$, both $F(k,1-k,1;\sin^2{(\sqrt{\xi}/2)})$ and $Z_J(\xi)$ are recessive solutions and $\phi_k\left(\sin^2{\left(\frac{\sqrt{\xi}}{2}\right)} \right)$ and $Z_Y(\xi)$ are dominant.
This is because 
\begin{equation}\label{phik and F at 0}
\phi_k(\xi)\sim \log{\xi}, \quad {}F(k,1-k,1;\xi)\sim 1 \text{ as } \xi \rightarrow 0,
\end{equation}
\begin{equation}
Z_Y(\xi)\xi^{-1/2}\sim \log{\xi}, \quad Z_J(\xi)\xi^{-1/2} \sim 1 \text{ as } \xi \rightarrow 0.
\end{equation}
For the definition and the theory of recessive and dominant solutions see \cite[$\S 5.7$]{O}.

Thus there is a constant $c_0$ such that
\begin{equation}\label{f1coeff}
\xi^{1/4}(\sin{(\sqrt{\xi})})^{1/2}F(k,1-k,1;\sin^2{(\sqrt{\xi}/2)})=c_0Z_J(\xi).
\end{equation}
Note that
\begin{equation}
\lim_{\xi \rightarrow 0}F(k,1-k,1;\sin^2{(\sqrt{\xi}/2)})=1.
\end{equation}
By \cite[Eq.~10.2.2]{HMF}  we have $J_0(x)=1+O(x^2)$ and $J_1(x)=x/2+O(x^3)$. Therefore,
\begin{equation}
Z_J(\xi)=\sqrt{\xi}\sum_{n=0}^{N}\frac{A_J(n;\xi)}{u^{2n}}+O(\xi) \text{ as } \xi \rightarrow 0.
\end{equation}
Choosing the constants of integration $\lambda_n$ in \eqref{recurrence1}  such that $A_J(0;0)=1$ and $A_J(n;0)=0$ for $n \geq 1$, we find
that $\lim_{\xi \rightarrow 0}Z_J(\xi)=\sqrt{\xi}$ and $c_0=1$.

Since $\phi_k\left(\sin^2{\left(\frac{\sqrt{\xi}}{2}\right)} \right)$ and $Z_Y(\xi)$ are dominant for any $0<\xi<\xi_2$, it is not possible to find a proportionality relation between them analogous to \eqref{f1coeff}.
To solve this problem, we apply the method described in \cite[$\S 12.5$]{O}.

The differential equation \eqref{diffurzxi} has two solutions $\phi_k(\xi)$ and $F(k,1-k,1;\xi)$, which are linearly independent in view of \eqref{phik and F at 0}. Therefore, $Z_Y$ can be written as a linear combination
\begin{multline}\label{zycoeff}
Z_{Y}(\xi)=\xi^{1/4}(\sin{(\sqrt{\xi})})^{1/2}
\left(\phi_k\left(\sin^2{\left(\frac{\sqrt{\xi}}{2}\right)}\right)c_1+F\left(k,1-k,1;\sin^2{\left(\frac{\sqrt{\xi}}{2}\right)}\right)c_2 \right)
\end{multline}
for some constants $c_1$, $c_2$.

Substituting  \eqref{f1coeff} with $c_0=1$ into \eqref{zycoeff} we have
\begin{equation}\label{eq:cycj}
\xi^{1/4}(\sin{(\sqrt{\xi})})^{1/2}\phi_k\left(\sin^2{\left(\frac{\sqrt{\xi}}{2}\right)}\right)=\frac{1}{c_1}Z_Y(\xi)-\frac{c_2}{c_1}Z_J(\xi),
\end{equation}
provided that $c_1 \neq 0$.

In order to determine the constants $c_1$, $c_2$, we compute $Z_{Y}(\xi)$  and its derivative at $\xi_2=\pi^2/4$. Applying Lemmas \ref{lem:derivF}, \ref{lem:phik12} and Corollary \ref{phi derivative 1/2}, we have
\begin{equation}
Z_{Y}(\xi_2)=\xi_{2}^{1/4}c_1\phi_k(1/2),
\end{equation}
\begin{equation}
Z_{Y}'(\xi_2)=\frac{Z_{Y}(\xi_2)}{4\xi_2}+\frac{c_2}{4\xi_{2}^{1/4}}\frac{\partial}{\partial x}F(k,1-k,1;x)\bigg|_{x=1/2}.
\end{equation}
Using Lemma \ref{lem:phik12} we find
\begin{equation}\label{eq:c1}
c_1=(-1)^{k/2}\frac{\Gamma(k/2+1/2)}{2\Gamma(1/2)\Gamma(k/2)}\frac{Z_Y(\xi_2)}{\xi_{2}^{1/4}},
\end{equation}
\begin{equation}\label{eq:c2}
c_2=(-1)^{k/2}\frac{\pi \Gamma(k/2)}{\Gamma(1/2)\Gamma(k/2+1/2)}\xi_{2}^{1/4}\left(Z_Y'(\xi_2)-\frac{Z_Y(\xi_2)}{4\xi_2}\right).
\end{equation}

The final step is to compute $Z_Y(\xi_2)$ and $Z_Y'(\xi_2)$.

\begin{lem}\label{lem:zy}
For $\xi_2=\pi^2 /4$ the following asymptotic formulas hold
\begin{equation}\label{eq:zy}
Z_Y(\xi_2)=\frac{(-1)^{k/2+1}}{\sqrt{u}}\left(1+\frac{1}{u^2}(\lambda_1-1/16)+O(u^{-4}) \right),
\end{equation}
\begin{equation}\label{eq:zyderriv}
Z_Y'(\xi_2)=
\frac{(-1)^{k/2+1}}{\sqrt{u}\pi^2}\left(1+\frac{1}{u^2}\left[ \frac{5\lambda_1}{4}-\frac{5}{64}-\frac{405}{128\pi^2}\right]\right)+O(u^{-5/2}).
\end{equation}
\end{lem}
\begin{proof}
It follows from  \cite[Eq. 3.9]{BD}  that
\begin{equation*}
\epsilon_{2N+1,1}(u;\xi_2)\rightarrow0 \text{ and }\frac{\partial}{\partial \xi}\epsilon_{2N+1,1}(u;\xi)\rightarrow0 \text{ as } \xi\rightarrow\xi_2.
\end{equation*}
Such properties of error terms in the Liouville-Green approximation are well known. For example, for a simpler differential equation the same property is stated in \cite[Eq. 2.19, p.196]{O}.
Therefore,
\begin{equation*}
Z_Y(\xi_2)=\sqrt{\xi_2}Y_0(u\sqrt{\xi_2})\sum_{n=0}^{N}\frac{A_Y(n;\xi_2)}{u^{2n}}-\frac{\xi_2}{u}Y_1(u\sqrt{\xi_2})\sum_{n=0}^{N-1}\frac{B_Y(n;\xi_2)}{u^{2n}}.
\end{equation*}
Using the relations \eqref{eq:derivW} and \eqref{eq:derivV}, we obtain
\begin{multline*}
Z_Y'(\xi_2)=\sqrt{\xi_2}Y_0(u\sqrt{\xi_2})\biggl(\frac{1}{2\xi_2}\sum_{n=0}^{N}\frac{A_Y(n;\xi_2)}{u^{2n}}+\sum_{n=0}^{N}\frac{A_Y'(n;\xi_2)}{u^{2n}}
- \frac{1}{2}\sum_{n=0}^{N-1}\frac{B_Y(n;\xi_2)}{u^{2n}} \biggr)\\-
\xi_2Y_1(u\sqrt{\xi_2})\biggl( \frac{u}{2\xi_2}\sum_{n=0}^{N}\frac{A_Y(n;\xi_2)}{u^{2n}}+\frac{1}{2\xi_2 u}\sum_{n=0}^{N-1}\frac{B_Y(n;\xi_2)}{u^{2n}}
+\frac{1}{u}\sum_{n=0}^{N-1}\frac{B_Y'(n;\xi_2)}{u^{2n}} \biggr).
\end{multline*}
For our purposes it is sufficient to take $N=1$.
Applying \cite[Eq.~10.17.1,~10.17.4]{HMF} and \cite[Eq.~8.451(1,7,8)]{GR}, we can write the Hankel expansions for the Bessel functions:
\begin{equation*}
\sqrt{\xi_2}Y_0(u\sqrt{\xi_2})=\frac{(-1)^{k/2+1}}{\sqrt{u}}\sum_{j=0}^{\infty}(-1)^{j}\frac{a_{2j}(0)}{(\pi u/2)^{2j}},
\end{equation*}
\begin{equation*}
\xi_2Y_1(u \sqrt{\xi_2})=\frac{\pi}{2}\frac{(-1)^{k/2+1}}{\sqrt{u}}\sum_{j=0}^{\infty}(-1)^j\frac{a_{2j+1}(1)}{(\pi u/2)^{2j+1}},
\end{equation*}
where
\begin{equation*}
a_j(v)=\frac{\Gamma(v+j+1/2)}{2^j j! \Gamma(v-j+1/2)}.
\end{equation*}
Thus
\begin{equation*}
\sqrt{\xi_2}Y_0(u\sqrt{\xi_2})=\frac{(-1)^{k/2+1}}{\sqrt{u}}\left(  1-\frac{a_2(0)}{(\pi u/2)^2}+O(u^{-4})\right),
\end{equation*}
\begin{equation*}
\xi_2Y_1(u \sqrt{\xi_2})=\frac{\pi}{2}\frac{(-1)^{k/2+1}}{\sqrt{u}}\left(\frac{a_1(1)}{\pi u/2}+O(u^{-3}) \right).
\end{equation*}
Consequently, taking $N=1$ we have
\begin{multline*}
Z_Y(\xi_2)=\frac{(-1)^{k/2+1}}{\sqrt{u}}\Biggl(A_Y(0;\xi_2)+
\frac{1}{u^2}\left(A_Y(1;\xi_2)-\frac{A_Y(0;\xi_2)a_2(0)}{(\pi/2)^2}-B_Y(0;\xi_2)a_1(1) \right)\\+O(u^{-4}) \Biggr).
\end{multline*}
According to \eqref{eq:bo} and \eqref{eq:a1}
\begin{equation*}
  A_Y(0;\xi_2)=1,\quad A_Y(1;\xi_2)=-\frac{1}{16}+\frac{15}{32\pi^2}+\lambda_1, \quad B_Y(0;\xi_2)=\frac{1}{2\pi^2}.
\end{equation*}
Substituting $a_2(0)=9/128$, $a_1(1)=3/8$, we prove the equation \eqref{eq:zy}.

Next we find an asymptotic expansion for the derivative of $Z_Y(\xi)$ at $\xi=\xi_2=\pi^2/4$. From the recurrence relations \eqref{recurrence1} and \eqref{recurrence2}, it follows that
\begin{equation*}
A_Y'(1;\xi_2)=-\frac{15}{8\pi^4}+\frac{3}{32\pi^2}.
\end{equation*}
Taking $N=1$, using the Hankel expansions for the Bessel functions and computing $a_3(1)=-105/1024$, we obtain the formula \eqref{eq:zyderriv}.
\end{proof}

\begin{lem}\label{lemcj}
For $\lambda_1=\frac{1}{16}+\frac{405}{32\pi^2}$ we have
\begin{equation}\label{eq:cjexpl}
C_J=-\pi^2\frac{\Gamma^2(k/2)}{\Gamma^2(k/2+1/2)}\frac{1}{Z_Y(\xi_2)}\left(Z_Y'(\xi_2)-\frac{Z_Y(\xi_2)}{\pi^2} \right),
\end{equation}
\begin{equation}\label{eq:cjappr}
C_J=O(k^{-5}).
\end{equation}
\end{lem}
\begin{proof}
Note that with our choice of $\lambda_1$ the value of $c_1$ is non-zero.
Therefore, by \eqref{eq:cycj2}, \eqref{eq:cycj}, \eqref{eq:c1}, \eqref{eq:c2} we have
\begin{equation*}
C_J=-\frac{c_2}{c_1}=-\pi^2\frac{\Gamma^2(k/2)}{\Gamma^2(k/2+1/2)}\frac{1}{Z_Y(\xi_2)}\left(Z_Y'(\xi_2)-\frac{Z_Y(\xi_2)}{\pi^2} \right).
\end{equation*}
By Lemma \ref{lem:zy}
\begin{equation*}
Z_Y'(\xi_2)-\frac{Z_Y(\xi_2)}{\pi^2}=\frac{(-1)^{k/2+1}}{u^{5/2}}\left(\frac{\lambda_1}{4\pi^2}-\frac{405}{128 \pi^4}-\frac{1}{64\pi^2} \right)+O(u^{-9/2}).
\end{equation*}
Since $\lambda_1=\frac{1}{16}+\frac{405}{32\pi^2}$  the first summand in the formula above is zero, so that $$Z_Y'(\xi_2)-\frac{Z_Y(\xi_2)}{\pi^2}=O(u^{-9/2}).$$
Stirling's formula and  \cite[Eq.~5.5.5]{HMF} yield
\begin{equation*}
\frac{\Gamma^2(k/2)}{\Gamma^2(k/2+1/2)}=\left( \frac{2^{k-1}\Gamma^2(k/2)}{\sqrt{\pi}\Gamma(k)}\right)^2=O(k^{-1})
\end{equation*}
Combining all results we find that $C_J=O(u^{-5})=O(k^{-5})$.
\end{proof}

\begin{lem}
For $\lambda_1=\frac{1}{16}+\frac{405}{32\pi^2}$, $u=k-1/2$ we have
\begin{equation}\label{eq:cyexp}
C_Y=(-1)^{k/2}\frac{2\Gamma(1/2)\Gamma(k/2)}{\Gamma(k/2+1/2)}\frac{\xi_{2}^{1/4}}{Z_Y(\xi_2)}.
\end{equation}
For any $n \geq 1$ there exist constants $d_1,d_2,\ldots, d_n$ such that
\begin{equation}\label{cyexpseries}
C_Y=-2\pi\left(1+\frac{d_1}{u}+\frac{d_2}{u^2}+\ldots\frac{d_n}{u^n}+O(u^{-n-1})\right).
\end{equation}
\end{lem}
\begin{proof}
Applying the formulas \eqref{eq:cycj2}, \eqref{eq:cycj}, \eqref{eq:c1}, we prove \eqref{eq:cyexp}.
Asymptotics \eqref{cyexpseries} follows from the equation \eqref{eq:zy}, Stirling's formula and  \cite[Eq.~5.5.5]{HMF}.
\end{proof}

Finally, we deduce the main result of this subsection.
\begin{thm}\label{thm:approxphi}
Let $\xi_2=\pi^2/4$. Then for any $\xi \in (0,\xi_2)$ the equality \eqref{eq:cycj2} holds with
$Z_Y$, $Z_J$, $C_Y$, $C_J$ given by \eqref{zyxi}, \eqref{eq:zjapprox}, \eqref{eq:cyexp}, \eqref{eq:cjexpl}, respectively.
\end{thm}
\subsection{Asymptotic approximation of $\psi_k$ and $\Phi_k$}\label{sub:approxpsik}

In this section we study the functions defined by \eqref{Phi_k2} and \eqref{psi_k} when $u=v=0$, namely
\begin{equation}
\psi_k(x)=\psi_k(x;0;0)=2\frac{\Gamma^2(k)}{\Gamma(2k)}x^kF(k,k,2k;-x),
\end{equation}
\begin{equation}
\Phi_k(x)=\Phi_k(x;0;0)=2\frac{\Gamma^2(k)}{\Gamma(2k)}x^kF(k,k,2k;x).
\end{equation}

The main difference with the previous subsection is that now the Liouville-Green approximation is based on the $K$-Bessel functions.
See Theorem \ref{thm:approxPhi} for the precise statement.

By \cite[Eq. 15.8.1]{HMF}
\begin{equation*}
\psi_k(x)=\frac{2\Gamma^2(k)}{\Gamma(2k)}\left( \frac{x}{1+x}\right)^k
F\left(k,k,2k;\frac{x}{1+x}\right)=\Phi_k\left(\frac{x}{1+x}\right).
\end{equation*}
Hence
\begin{equation}\label{psiphi}
\psi_k\left(\frac{l}{n}\right)=\Phi_k\left(\frac{l}{n+l} \right),
\end{equation}
and it is sufficient to consider only the function $\Phi_k$.

Let $u:=k-1/2$ and
\begin{equation}\label{newf}
f(x):=\frac{1}{x^2(1-x)}, \quad g(x):=-\frac{1}{4x^2(1-x)^2}+\frac{1}{4x(1-x)}.
\end{equation}

\begin{lem}
The function $y(x)=F(k,k,2k;x)x^k\sqrt{1-x}$ is a solution of the differential equation
\begin{equation}\label{eq:diffurf21}
y''(x)-(u^2f(x)+g(x))y(x)=0.
\end{equation}
\end{lem}
\begin{proof}
The hypergeometric function $F(x)=F(k,k,2k;x)$ satisfies
\begin{equation*}
x(1-x)F''(x)+(2k-(2k+1)x)F'(x)-k^2F(x)=0.
\end{equation*}
Applying Lemma \ref{diffurabc} with $$\alpha(x)=x^{-k}(1-x)^{-1/2},$$ we have
\begin{equation*}
y''(x)+\biggl( \frac{1-(2k-1)^2}{4x^2}+\frac{1}{4(1-x)^2}+\frac{1-(2k-1)^2}{4x(1-x)}\biggr)y(x)=0.
\end{equation*}
The assertion follows by rearranging the expression in brackets.
\end{proof}

Note that the differential equations \eqref{eq:diffurf21}  and \eqref{diffurufg} are almost identical except that the functions $f(x)$ defined by \eqref{newf} and \eqref{oldf}  differ in sign. According to   \cite[Chapter~12]{O} this means  that in the current case the $I$ and $K$ Bessel functions (instead of $Y$ and $J$) should be chosen as approximation functions. Consequently, we transform \eqref{eq:diffurf21} to the type
\begin{equation}\label{eq:diffurZpsi}
\frac{d^2Z}{d\xi^2}+\left[ -\frac{u^2}{4\xi}+\frac{1}{4\xi^2}-\frac{\psi(\xi)}{\xi}\right]Z=0,
\end{equation}
where
\begin{equation}
\psi(\xi)=\frac{1}{16}\left(\frac{1}{\xi}-\frac{1}{\sinh^2{(\sqrt{\xi})}} \right).
\end{equation}
This can be done similarly to the previous case by making the change
\begin{equation}
Z(x):=\frac{y(x)}{\alpha(x)}, \quad \alpha(x):=\frac{(x^2-x^3)^{1/4}}{2(\artanh{(\sqrt{1-x})})^{1/2}}
\end{equation}
and the substitution
\begin{equation}
\xi:=4 \artanh^2{(\sqrt{1-x})}.
\end{equation}
Note that as $\xi \rightarrow 0$ the function $\psi(\xi)$ is analytic.
Removing the term with $\psi(\xi)/\xi$ in \eqref{eq:diffurZpsi} we obtain
\begin{equation}
Z''+\left(-\frac{u^2}{4\xi}+\frac{1}{4\xi} \right)Z=0.
\end{equation}
Solutions of this equation are given by (see \cite[Eq.~10.13.2]{HMF})
\begin{equation}
Z_L=\sqrt{\xi}L_0(u\sqrt{\xi}),
\end{equation}
where $L_0$ is either the $K_0$ or $I_0$ Bessel function.
In general,
\begin{equation}
L_v:=\begin{cases}
I_v\\e^{\pi i v}K_v
\end{cases}.
\end{equation}
Therefore, according to \cite[Eq.~2.09,~Chapter~12]{O}, a solution of the differential equation \eqref{eq:diffurZpsi} can be found in the form
\begin{equation}\label{eq:solZ}
Z_L=\sqrt{\xi}L_0(u\sqrt{\xi})\sum_{n=0}^{\infty}\frac{A(n;\xi)}{u^{2n}}+\frac{\xi}{u}L_1(u\sqrt{\xi})
\sum_{n=0}^{\infty}\frac{B(n;\xi)}{u^{2n}}.
\end{equation}
To determine the coefficients $A(n;\xi)$ and $B(n;\xi)$ we introduce two functions:
\begin{equation}
W(\xi):=\sqrt{\xi}L_0(u\sqrt{\xi}),\quad V(\xi):=\xi L_1(u\sqrt{\xi}),
\end{equation}
which satisfy the following differential equations (see \cite[Eq.~10.13.2,~10.13.5,~10.36]{HMF})
\begin{equation}
W''+\left( -\frac{u^2}{4\xi}+\frac{1}{4\xi^2}\right)W=0,
\end{equation}
\begin{equation}
V''-\frac{1}{\xi}V'+\left( -\frac{u^2}{4\xi}+\frac{3}{4\xi^2}\right)V=0.
\end{equation}
Furthermore, using  \cite[Eq. 10.29.2,10.29.3]{HMF} we prove that
\begin{equation}
V'=\frac{1}{2\xi}V+\frac{u}{2}W, \quad W'=\frac{1}{2\xi}W+\frac{u}{2\xi}V.
\end{equation}

Substituting the equation \eqref{eq:solZ} into \eqref{eq:diffurZpsi} we have
\begin{equation}
W(\xi)\sum_{n=0}^{\infty}\frac{C(n;\xi)}{u^{2n}}+V(\xi)\sum_{n=0}^{\infty}\frac{D(n;\xi)}{u^{2n+1}}=0,
\end{equation}
where
\begin{equation*}
C(n;\xi)=A''(n;\xi)+\frac{A'(n;\xi)}{\xi}-\frac{\psi(\xi)}{\xi}A(n;\xi)+B'(n;\xi)+\frac{B(n;\xi)}{2\xi},
\end{equation*}
\begin{equation*}
D(n;\xi)=B''(n-1;\xi)+\frac{B'(n-1;\xi)}{\xi}-\frac{\psi(\xi)}{\xi}B(n-1;\xi)+\frac{A'(n;\xi)}{\xi}.
\end{equation*}

Setting $C(n;\xi)=D(n;\xi)=0$ we find the recurrence relations:
\begin{equation}\label{rec:bnxi}
\sqrt{\xi}B(n;\xi)=-\sqrt{\xi}A'(n;\xi)+\int_{0}^{\xi}\left(\psi(x) A(n;x)-\frac{1}{2}A'(n;x)\right)\frac{dx}{\sqrt{x}},
\end{equation}
\begin{equation}
A(n;\xi)=-\xi B'(n-1;\xi)+\int_{0}^{\xi}\psi(x)B(n-1;x)dx+\lambda_n
\end{equation}
for some real constants of integration $\lambda_n$.

Let $A(0;\xi)=1$. Then
\begin{equation}\label{B0def}
B(0;\xi)=\frac{1}{8}\left( \frac{\coth{\sqrt{\xi}}}{\sqrt{\xi}}-\frac{1}{\xi}\right),
\end{equation}

\begin{multline}
A(1;\xi)=-\frac{1}{8}\left( \frac{1}{\xi}-\frac{\coth{\sqrt{\xi}}}{2\sqrt{\xi}}-\frac{1}{2\sinh^2{(\sqrt{\xi})}}\right)+
\frac{1}{128}\left( \coth{\sqrt{\xi}}-\frac{1}{\sqrt{\xi}}\right)^2+\lambda_1.
\end{multline}

Note that
\begin{equation}
B(0;\xi)=\frac{1}{24}+O(\xi), \quad A(1;\xi)=\lambda_1+O(\xi) \text{ as } \xi \rightarrow 0,
\end{equation}
\begin{equation}
\lim_{\xi \rightarrow \infty}\sqrt{\xi}B(0;\xi)=\frac{1}{8}, \quad \lim_{\xi \rightarrow \infty}A(1;\xi)=\frac{1}{128}+\lambda_1.
\end{equation}

\begin{thm}\label{thm:zk}
For each value of $u$ and each nonnegative integer $N$, the equation \eqref{eq:diffurZpsi} has the solution $Z_K(\xi)$ which is infinitely differentiable in $\xi$ on the interval $(0, \infty)$ and is given by
\begin{multline}\label{eq:zkxi}
Z_K(\xi)=\sqrt{\xi}K_0(u\sqrt{\xi})\sum_{n=0}^{N}\frac{A_K(n;\xi)}{u^{2n}}-\frac{\xi}{u}K_1(u\sqrt{\xi})
\sum_{n=0}^{N-1}\frac{B_K(n;\xi)}{u^{2n}}+\epsilon_{2N+1,3}(u,\xi),
\end{multline}
where
\begin{equation*}
|\epsilon_{2N+1,3}(u,\xi)|\leq  \frac{\sqrt{\xi}K_0(u\sqrt{\xi})}{u^{2N+1}}V_{\xi,\infty}(\sqrt{\xi}B_K(N;\xi))\exp\left( \frac{1}{u}V_{\xi, \infty}(\sqrt{\xi}B_K(0;\xi))\right).
\end{equation*}
In particular, for $N=1$
\begin{equation}\label{error in K LG 1}
\epsilon_{3,3}(u,\xi)\ll \frac{\sqrt{\xi}K_0(u\sqrt{\xi})}{u^{3}}\min\left(\sqrt{\xi}, \frac{1}{\xi}\right).
\end{equation}
\end{thm}
\begin{proof}
Our arguments require only minor changes comparing to \cite[Theorem 3.1, p.441]{O}. The first difference is that we take the parameter $\beta=\infty$ in \cite[Theorem 3.1,  p.441]{O}. As mentioned on \cite[p.442]{O}, this is possible as long as the variations of $\sqrt{x}B(0;x)$ and $\sqrt{x}B(n;x)$ converge at infinity. The convergence of variations of $\sqrt{x}B(0;x)$ follows directly from \eqref{B0def}.
Furthermore, note that
\begin{equation*}
\psi^{(s)}(\xi)=O\left(\frac{1}{|\xi|^{s+1}} \right).
\end{equation*}
Hence for $n>1$ the variation
\begin{equation*}
\Var_{\xi,\infty}(\sqrt{x}B(n;x))=\int_{\xi}^{\infty}|(\sqrt{x}B(n;x))'|dx
\end{equation*}
converges by \cite[Exercise~4.2,~p.445]{O}.
To prove \eqref{error in K LG 1} we need to estimate the variation of $\sqrt{x}B(1;x)$ at infinity.
Using the recurrence relation \eqref{rec:bnxi}, we find
\begin{equation}
(\sqrt{x}B(1;x))'=-\frac{1}{\sqrt{x}}\left( xA''(1;x)+A'(1;x)-\psi(x)A(1;x)\right).
\end{equation}
Therefore,
\begin{equation}
(\sqrt{x}B(1;x))'=O(x^{-1/2}) \text{ as } x \rightarrow 0
\end{equation}
and
\begin{equation}
(\sqrt{x}B(1;x))'=O(x^{-2}) \text{ as } x \rightarrow \infty.
\end{equation}
Thus
\begin{equation*}
\Var_{\xi,\infty}(\sqrt{x}B(1;x))=\int_{\xi}^{\infty}|(\sqrt{x}B(1;x))'|dx\ll\min\left(\sqrt{\xi}, \frac{1}{\xi}\right).
\end{equation*}

\end{proof}


The solution of the differential equation \eqref{eq:diffurZpsi}
\begin{multline}
Z(\xi)=\biggl(\Phi_k(x)\frac{\sqrt{1-x}}{\alpha(x)}\biggr)\bigg|_{x=1/\cosh^2{\left(\sqrt{\xi}/2\right)}}=\\
\Phi_k\left(\frac{1}{\cosh^2{\left(\sqrt{\xi}/2\right)}} \right)\left( \xi\sinh^2{\left(\sqrt{\xi}\right)}\right)^{1/4}
\end{multline}
is recessive as $\xi \rightarrow \infty$. Another recessive solution is $Z_K(\xi)$ defined by \eqref{eq:solZ}
with $L_v=e^{\pi i v}K_v$. Therefore, there exists $C_K=C_K(u)$ such that
\begin{equation}\label{eq:phiklim}
\Phi_k\left(\frac{1}{\cosh^2{\left(\sqrt{\xi}/2\right)}} \right)\left( \xi\sinh^2{\left(\sqrt{\xi}\right)}\right)^{1/4}=
C_KZ_K(\xi).
\end{equation}
Computing the limit as $\xi \rightarrow \infty$ of the left and right-hand sides of the equation \eqref{eq:phiklim}, we find
\begin{equation}\label{eq:ckexpl}
C_K=2\frac{\Gamma^2(k)}{\Gamma(2k)}\frac{2^{2k}\sqrt{u}}{\sqrt{\pi}}\left[\sum_{n=0}^{N}\frac{a_n}{u^{2n}}-\sum_{n=0}^{N-1}\frac{b_n}{u^{2n+1}}\right]^{-1},
\end{equation}
where
\begin{equation}
a_n=\lim_{\xi \rightarrow \infty}A(n;\xi), \quad b_n=\lim_{\xi \rightarrow \infty}B(n;\xi)\sqrt{\xi},
\end{equation}
\begin{equation}
a_0=1, \quad a_1=\frac{1}{128}+\lambda_1, \quad b_0=\frac{1}{8}.
\end{equation}
Since
\begin{equation*}
\frac{\Gamma^2(k)}{\Gamma(2k)}=\frac{2\sqrt{\pi}}{\sqrt{k}2^{2k}}(1+O(k^{-1})),
\end{equation*}
we have
\begin{equation}\label{asymp:ck}
C_K=4+O(k^{-1}).
\end{equation}

To sum up, we proved the following result.
\begin{thm}\label{thm:approxPhi}
For  $\xi \in (0, \infty)$ the equality \eqref{eq:phiklim} holds with $Z_K$, $C_K$ given by \eqref{eq:zkxi}, \eqref{eq:ckexpl} , respectively.
\end{thm}
\section{Error terms for the individual weight}

\begin{lem}\label{M2individual}
The following exact formula holds
\begin{multline}
M_2(l;0,0)=\sum_{f \in H_{2k}(1)}^{h}\lambda_f(l)L_{f}^{2}(1/2)=(1+(-1)^k) 
\Biggl(\frac{\tau(l)}{\sqrt{l}}\left[ 2\frac{\Gamma'}{\Gamma}(k)-\log{l}-2\log(2\pi )+2\gamma \right]\\+
\frac{1}{2\sqrt{l}}\sum_{n=1}^{l-1}\tau(n)\tau(l-n)\phi_k\left( \frac{n}{l}\right)+
\frac{1}{\sqrt{l}}\sum_{n=1}^{\infty}\tau(n)\tau(n+l)\Phi_k\left( \frac{l}{n+l}\right)\Biggl).
\end{multline}
\end{lem}
\begin{proof}
The assertion follows from Theorem \ref{thm:kuznetsov} by computing the limit as $u\rightarrow 0$, $v \rightarrow 0$.
The off-diagonal terms can be simplified as follows.
First, using \eqref{funceqphi} we find that
\begin{equation*}
\sum_{n=1}^{l-1}\tau(n)\tau(l-n)\phi_k\left( \frac{n}{l}\right)=(-1)^k\sum_{n=1}^{l-1}\tau(n)\tau(l-n)\phi_k\left( \frac{n}{l}\right).
\end{equation*}
Second, by  \eqref{psiphi} we have
\begin{equation*}
\psi_k\left(\frac{l}{n}\right)=\Phi_k\left(\frac{l}{n+l} \right).
\end{equation*}
Therefore,
\begin{multline*}
\frac{1}{\sqrt{l}}\sum_{n = l+1}^{\infty}\tau(n)\tau(n-l)\Phi_k\left(\frac{l}{n} \right)+
\frac{(-1)^k}{\sqrt{l}}\sum_{n=0}^{\infty}\tau(n)\tau(n+l)\psi_k\left(\frac{l}{n} \right)=\\
\frac{1+(-1)^k}{\sqrt{l}}\sum_{n=1}^{\infty}\tau(n)\tau(n+l)\Phi_k\left( \frac{l}{n+l}\right).
\end{multline*}
\end{proof}

\begin{lem}\label{lem:errorpsi} For any $\epsilon>0$, $l \ll k^2$ we have
\begin{equation}
\frac{1}{\sqrt{l}}\sum_{n=1}^{\infty}\tau(n)\tau(n+l)\Phi_{2k}\left( \frac{l}{n+l}\right)\ll \exp\left(-\frac{ck}{\sqrt{l}}\right)
\left( \frac{l^{\epsilon}}{l^{1/4}k^{1/2}}+\frac{l^{1/2+\epsilon}}{k^{3/2}}\right)
\end{equation}
for some absolute constant $c>0$.
\end{lem}

\begin{proof}
Consider
\begin{multline*}
\frac{1}{\sqrt{l}}\sum_{n=1}^{\infty}\tau(n)\tau(n+l)\Phi_{2k}\left( \frac{l}{n+l}\right)\ll
\frac{l^{\epsilon}}{\sqrt{l}}\sum_{n=1}^{\infty}n^{\epsilon}\left|\Phi_{2k}\left( \frac{l}{n+l}\right)\right|\ll
\\ \frac{l^{\epsilon}}{\sqrt{l}}\int_{1}^{\infty}x^{\epsilon}\left|\Phi_{2k}\left( \frac{l}{x+l}\right)\right|dx+
O\left( \frac{l^{\epsilon}}{\sqrt{l}}\Phi_{2k}\left( \frac{l}{1+l}\right)\right).
\end{multline*}

First, we estimate the integral 
\begin{equation*}
I:=\frac{l^{\epsilon}}{\sqrt{l}}\int_{1}^{\infty}x^{\epsilon}\left|\Phi_{2k}\left( \frac{l}{x+l}\right)\right|dx.
\end{equation*}
Letting
\begin{equation*}
x:=l\sinh^2{\left(\frac{\sqrt{\xi}}{2}\right)},\quad a:=2\arcsinh{\left(\frac{1}{\sqrt{l}}\right)},
\end{equation*}
we have
\begin{equation*}
I=\frac{l^{\epsilon}}{\sqrt{l}}l\int_{a^{2}}^{\infty}\sinh^{\epsilon}{\left(\frac{\sqrt{\xi}}{2}\right)}\left|\Phi_{2k}\left( \frac{1}{\cosh^2\left(\frac{\sqrt{\xi}}{2}\right)}\right)\right|\frac{\sinh\left(\sqrt{\xi}\right)}{\sqrt{\xi}}d\xi.
\end{equation*}

Applying the Liouville-Green method (see \eqref{eq:phiklim}) the integral is equal to
\begin{equation*}
I=l^{1/2+\epsilon}C_K\int_{a^{2}}^{\infty}\sinh^{\epsilon}{\left(\frac{\sqrt{\xi}}{2}\right)}\left|Z_K(\xi)\right|\frac{(\sinh(\sqrt{\xi}))^{1/2}}{\xi^{3/4}}d\xi,
\end{equation*}
where $C_K$ satisfies the asymptotic formula \eqref{asymp:ck}.
Note that $$\sqrt{\xi}\geq 2 \arcsinh{(l^{-1/2})}\gg 1/\sqrt{l}.$$ Thus
\begin{equation*}(2k-1/2)\sqrt{\xi}\gg k/\sqrt{l}\gg 1\text{ for any }l \ll k^2.
\end{equation*}
Applying Theorem \ref{thm:zk} with $N=0$ we estimate
\begin{equation*}
Z_K(\xi)\ll \sqrt{\xi}K_0((2k-1/2)\sqrt{\xi}). 
\end{equation*}
The last inequality and \cite[Eq.~10.40.2]{HMF} yield
\begin{multline*}
I \ll l^{1/2+\epsilon}\int_{a^{2}}^{\infty}\left|K_0((2k-1/2)\sqrt{\xi})\right|\frac{(\sinh(\sqrt{\xi}))^{1/2+\epsilon}}{\xi^{1/4}}d\xi \ll\\
\frac{l^{1/2+\epsilon}}{\sqrt{2k-1/2}}\int_{a^{2}}^{\infty}\exp(-(2k-1/2)\sqrt{\xi})\frac{(\sinh(\sqrt{\xi}))^{1/2+\epsilon}}{\sqrt{\xi}}d\xi.
\end{multline*}
Making the change of variables $\xi=x^2$ and splitting the integral into two parts
\begin{multline*}
I\ll \frac{l^{1/2+\epsilon}}{\sqrt{k}}\int_{a}^{\infty}\exp(-(2k-1/2)x)(\sinh (x))^{1/2+\epsilon}dx\ll \\
\frac{l^{1/2+\epsilon}}{\sqrt{k}}\biggl( \int_{a}^{1}\exp(-(2k-1/2)x)x^{1/2+\epsilon}dx+\\
\int_{1}^{\infty}\exp(-(2k-1/2)x)\exp(x(1/2+\epsilon))dx\biggr)\ll \frac{l^{1/2+\epsilon}}{k^{3/2}}\exp\left(-\frac{ck}{\sqrt{l}}\right)
\end{multline*}
for some $c>0$.

Now we estimate the second term
\begin{equation*}
E:=\frac{l^{\epsilon}}{\sqrt{l}}\Phi_{2k}\left( \frac{l}{n+l}\right).
\end{equation*}
Let
\begin{equation*}
\xi_0:=\left( 2\artanh{\left(\frac{1}{\sqrt{l+1}}\right)}\right)^2.
\end{equation*}
Then by equation \eqref{eq:phiklim}
\begin{equation*}
E=\frac{l^{\epsilon}}{\sqrt{l}}C_K\frac{Z_K(\xi_0)}{(\xi_0\sinh^2{(\sqrt{\xi_0})})^{1/4}}.
\end{equation*}
Note that $\xi_0\sim l^{-1}$. Then for $l \ll k^2$ we have
\begin{equation*}
Z_K(\xi_0)\ll \sqrt{\xi_0}K_0((2k-1/2)\sqrt{\xi_0})\ll \frac{\xi_{0}^{1/4}}{\sqrt{k}}\exp{(-(2k-1/2)\sqrt{\xi_0})}.
\end{equation*}
Finally, for some constant $c>0$ independent of $k$
\begin{equation*}
E\ll \frac{l^{\epsilon}}{\sqrt{l}}\frac{\exp{(-(2k-1/2)\sqrt{\xi_0})}}{\sqrt{k}(\sinh(\sqrt{\xi_0}))^{1/2}}
\ll \frac{l^{\epsilon}}{l^{1/4}k^{1/2}}\exp\left(-\frac{ck}{\sqrt{l}}\right).
\end{equation*}

\end{proof}

\begin{lem}For any $\epsilon>0$, $l\ll k^2$ we have
\begin{equation}
\frac{1}{\sqrt{l}}\sum_{n=1}^{l-1}\tau(n)\tau(n-l)\phi_{2k}\left( \frac{n}{l}\right)\ll \frac{l^{1/2+\epsilon}}{\sqrt{k}}.
\end{equation}
\end{lem}
\begin{proof}
We apply \eqref{eq:cycj2} with $n/l=:\sin^2\left(\frac{\sqrt{\xi}}{2}\right),$
so that
\begin{equation*}
\phi_{2k}\left(\frac{n}{l}\right)=\frac{C_YZ_Y(\xi)+C_JZ_J(\xi)}{(2\arcsin{(\sqrt{n/l})})^{1/2}(2n/l)^{1/4}(1-n/l)^{1/4}}.
\end{equation*}
Using the property \eqref{funceqphi}  we find
\begin{equation*}
\frac{1}{\sqrt{l}}\sum_{n=1}^{l-1}\tau(n)\tau(n-l)\phi_{2k}\left( \frac{n}{l}\right)\ll
\frac{l^{\epsilon}}{\sqrt{l}}\sum_{n\leq l/2}\bigg|\phi_{2k}\left(\frac{n}{l}\right)\bigg|.
\end{equation*}
Thus $n/l \leq 1/2$ and $\xi\leq \pi^2/4$. Since $n/l\geq 1/l$ we have $\xi \gg 1/l$. Hence
\begin{equation*}
(2k-1/2)\sqrt{\xi}\gg k/\sqrt{l}\gg 1 \text{ when } l \ll k^2.
\end{equation*}
Applying Theorem  \ref{LGphi} with $N=0$, Theorem \ref{LGphi2} and standard estimates for the Bessel functions, we find
\begin{equation*}
C_YZ_Y(\xi)\ll \sqrt{\xi}Y_0((2k-1/2)\sqrt{\xi})\ll \frac{\xi^{1/4}}{k^{1/2}},
\end{equation*}
\begin{equation*}
C_JZ_J(\xi)\ll C_JJ_0((2k-1/2)\sqrt{\xi})\ll\frac{\xi^{1/4}}{k^{11/2}}
\end{equation*}
since $C_J=O(k^{-5})$ by Lemma \ref{lemcj}.
Consequently,
\begin{equation*}
\phi_{2k}\left(\frac{n}{l}\right)\ll \frac{1}{\sqrt{k}(n/l)^{1/4}(1-n/l)^{1/4}}.
\end{equation*}
Finally,
\begin{equation*}
\frac{l^{\epsilon}}{\sqrt{l}}\sum_{n\leq l/2}\bigg|\phi_{2k}\left(\frac{n}{l}\right)\bigg|\ll
\frac{l^{\epsilon}}{\sqrt{l}}\sum_{n\leq l/2}\frac{(l/n)^{1/4}}{\sqrt{k}}\ll \frac{l^{1/2+\epsilon}}{\sqrt{k}}.
\end{equation*}
\end{proof}

Combining all results together we prove the following asymptotic formula.
\begin{thm}\label{secondmoment}For any $\epsilon>0$, $l\ll k^2$, $k\equiv 0\Mod{2}$ we have
\begin{equation}
M_2(l;0,0)=\frac{2\tau(l)}{\sqrt{l}}(2\log{k}-\log{l}-2\log{2\pi}+2\gamma)+O\left( \frac{l^{1/2+\epsilon}}{\sqrt{k}}\right).
\end{equation}
\end{thm}
\section{Error terms on average}
The Liouville-Green method allows also proving sharp asymptotic formulas on average, as we show in this section. The main goal is to provide an alternative proof for the result \eqref{eq:onaverage} by Iwaniec and Sarnak.
To this end, we estimate the error terms averaged over $k$ with a test function $h \in C_{0}^{\infty}(\R^{+})$, which is non-negative, compactly supported on $[\theta_1,\theta_2]$ such that $\theta_2>\theta_1>0$ and
\begin{equation}
\|h^{(n)}\|_1\ll 1 \text{ for all }n \geq 0.
\end{equation}
Let
\begin{equation}
H:=\int_{0}^{\infty}h(y)dy,\quad H_1:=\int_{0}^{\infty}h(y)\log{y}dy.
\end{equation}
We denote the averaged moments as follows
\begin{equation}
A_1(l):=\sum_{k}h\left( \frac{4k}{K}\right)\sum_{f \in H_{4k}(1)}^{h}\lambda_f(l)L_{f}(1/2),
\end{equation}
\begin{equation}
A_2(l):=\sum_{k}h\left( \frac{4k}{K}\right)\sum_{f \in H_{4k}(1)}^{h}\lambda_f(l)L^{2}_{f}(1/2).
\end{equation}

\begin{lem} The following exact formula holds
\begin{multline}
A_2(l)=\frac{2\tau(l)}{\sqrt{l}}\frac{HK}{4}
\left(2\log{K}-\log{l}-2\log{8\pi}+2\gamma+2\frac{H_1}{H}\right)+\\O\left( \frac{1}{\sqrt{l}}\right)+
2\sum_{k}h\left( \frac{4k}{K}\right)\frac{1}{\sqrt{l}}\sum_{n=1}^{\infty}\tau(n)\tau(n+l)\Phi_{2k}\left( \frac{l}{n+l}\right)+\\
\sum_{k}h\left( \frac{4k}{K}\right)\frac{1}{\sqrt{l}}\sum_{n=1}^{l-1}\tau(n)\tau(l-n)\phi_{2k}\left( \frac{n}{l}\right).
\end{multline}
\end{lem}
\begin{proof}
We average over $k$ the result of Lemma \ref{M2individual}. To compute the main term, we use \cite[Eq.~5.11.2]{HMF}, namely
\begin{equation*}
\frac{\Gamma'}{\Gamma}(k)\sim \log{k}-\frac{1}{2k}-\sum_{r=1}^{\infty}\frac{B_{2r}}{2rk^{2r}},
\end{equation*}
where $B_{2r}$ are the Bernoulli numbers.
Note that
\begin{equation*}
\sum_{k}h\left( \frac{4k}{K}\right)\frac{1}{k}\ll 1.
\end{equation*}
By Poisson's summation formula we have
\begin{multline*}
2\sum_{k}h\left( \frac{4k}{K}\right)\log{k}=2\sum_{n}\int_{-\infty}^{\infty}h\left( \frac{4x}{K}\right)\log{x}
\exp(-2\pi inx)dx=\\
2\frac{K}{4}\sum_{n}\int_{-\infty}^{\infty}h(y)(\log{y}+\log{K}-\log{4})\exp\left(\frac{-2\pi inyK}{4}\right)dy.
\end{multline*}
Note that in the last expression the summand  corresponding to $n=0$ is equal to
\begin{equation*}
2\left(\frac{HK}{4}(\log{K}-\log{4})+\frac{H_1K}{4}\right).
\end{equation*}
If $n \neq 0$ we integrate by parts $a\geq 2$ times and estimate the expression by its absolute value, obtaining
\begin{equation*}
\int_{-\infty}^{\infty}\frac{\partial^a}{\partial y^a}\biggl( h(y)\log\frac{yK}{4}\biggr)
\frac{1}{(nK)^a}dy \ll \frac{\log{K}}{(nK)^a}.
\end{equation*}
Similarly,
\begin{equation*}
\sum_{k}h\left( \frac{4k}{K}\right)=\frac{HK}{4}+O\left(\frac{1}{K^a}\right).
\end{equation*}

\end{proof}

\begin{lem}
For $l\ll K^{2-\epsilon}$ and any $A>0$ the following estimate holds
\begin{equation}
\sum_{k}h\left( \frac{4k}{K}\right)\frac{1}{\sqrt{l}}\sum_{n=1}^{\infty}\tau(n)\tau(n+l)\Phi_{2k}\left( \frac{l}{n+l}\right)\ll K^{-A}.
\end{equation}
\end{lem}
\begin{proof}
Averaging the result of Lemma \ref{lem:errorpsi} we obtain
\begin{multline*}
\sum_{k}h\left( \frac{4k}{K}\right)\frac{1}{\sqrt{l}}\sum_{n=1}^{\infty}\tau(n)\tau(n+l)\Phi_{2k}\left( \frac{l}{n+l}\right)\ll \\
l^{\epsilon}\sum_{k}h\left( \frac{4k}{K}\right) \left(\frac{1}{l^{1/4}k^{1/2}}+\frac{l^{1/2}}{k^{3/2}} \right)\exp{(-ck/\sqrt{l})}
\ll K^{-A}.
\end{multline*}

\end{proof}

\begin{lem}
For any $\epsilon>0$, for any $a\geq 2$ and for $l\ll K^{2-\epsilon}$ the following estimate holds
\begin{equation*}
\sum_{k}h\left( \frac{4k}{K}\right)\frac{1}{\sqrt{l}}\sum_{n=1}^{l-1}\tau(n)\tau(l-n)\phi_{2k}\left( \frac{n}{l}\right)\ll
 \frac{l^{a/2-1/4+\epsilon}}{K^{a+1/2}}K+\frac{l^{\epsilon}}{\sqrt{lK}}K+\frac{l^{1/2+\epsilon}}{K^{7/2}}K.
\end{equation*}

\end{lem}
\begin{proof}
According to \eqref{funceqphi} we have
\begin{equation*}
\phi_{2k}\left( \frac{n}{l}\right)=\phi_{2k}\left( \frac{l-n}{l}\right).
\end{equation*}
If $l$ is odd, then
\begin{equation*}
\sum_{n=1}^{l-1}\tau(n)\tau(l-n)\phi_{2k}\left( \frac{n}{l}\right)=2\sum_{n=1}^{l/2}\tau(n)\tau(l-n)\phi_{2k}\left( \frac{n}{l}\right),
\end{equation*}
and if $l $ is even we have
\begin{equation*}
\sum_{n=1}^{l-1}\tau(n)\tau(l-n)\phi_{2k}\left( \frac{n}{l}\right)=2\sum_{n=1}^{l/2}\tau(n)\tau(l-n)\phi_{2k}\left( \frac{n}{l}\right)-\phi_{2k}(1/2)\tau^2(l/2).
\end{equation*}
Contribution of $\phi_{2k}(1/2)\tau^2(l/2)$ can be estimated using the formula \eqref{phi12}:
\begin{multline*}
\frac{\tau^2(l/2)}{\sqrt{l}}\sum_{k}h\left( \frac{4k}{K}\right)\phi_{2k}(1/2)=
(-1)^k\frac{\tau^2(l/2)}{\sqrt{l}}\sum_{k}h\left( \frac{4k}{K}\right)\frac{2\sqrt{\pi}\Gamma(k)}{\Gamma(k+1/2)}\\\ll
l^{-1/2+\epsilon}\sum_{k}h\left( \frac{4k}{K}\right)\frac{1}{\sqrt{k}}\ll l^{-1/2+\epsilon}\frac{1}{\sqrt{K}}\sum_{k}h\left( \frac{4k}{K}\right)\ll \frac{l^{\epsilon}}{l^{1/2}K^{1/2}}K.
\end{multline*}
Next we estimate
\begin{equation*}\sum_{k}h\left( \frac{4k}{K}\right)\frac{1}{\sqrt{l}}\sum_{n=1}^{l/2}\tau(n)\tau(l-n)\phi_{2k}\left( \frac{n}{l}\right)\ll
\frac{l^{\epsilon}}{l^{1/2}}\sum_{n=1}^{  l/2}\left|\sum_{k}h\left( \frac{4k}{K}\right)\phi_{2k}\left(\frac{n}{l} \right)\right|.
\end{equation*}
Applying the formula \eqref{eq:cycj2}  we have
\begin{equation*}
\phi_{2k}\left(\frac{n}{l}\right)=\frac{C_YZ_Y(4\arcsin^2(\sqrt{n/l}))+C_JZ_J(4\arcsin^2(\sqrt{n/l}))}{(2\arcsin{(\sqrt{n/l})})^{1/2}(2n/l)^{1/4}(1-n/l)^{1/4}}.
\end{equation*}
Let $u:=2k-1/2$, $\xi:=4\arcsin^2(\sqrt{n/l})$. Then using \eqref{eq:zjapprox} and \eqref{eq:cjappr} we obtain
\begin{equation*}
C_JZ_J(\xi)=O\left( \frac{\sqrt{\xi}J_0(u\sqrt{\xi})}{k^5}\right).
\end{equation*}
Since $u\sqrt{\xi}=(4k-1)\arcsin(\sqrt{n/l})\gg K/\sqrt{l}\gg 1$ when $l\ll K^{2-\epsilon}$, we have
\begin{equation*}
J_0(u\sqrt{\xi})\ll \frac{1}{(u\sqrt{\xi})^{1/2}}\ll \frac{1}{\sqrt{k}(\arcsin(\sqrt{n/l}))^{1/2}}.
\end{equation*}
Therefore, the contribution of the term with $C_JZ_J$ is bounded by
\begin{multline*}
\frac{l^{\epsilon}}{\sqrt{l}}\sum_{n=1}^{l/2}\frac{(l/n)^{1/4}}{(\arcsin(\sqrt{n/l}))^{1/2}}\sum_{k}h\left( \frac{4k}{K}\right) \frac{\arcsin(\sqrt{n/l})|J_0(u\sqrt{\xi})|}{k^5}\ll\\
\frac{l^{\epsilon}}{\sqrt{l}}\sum_{n=1}^{l/2}\left(\frac{l}{n}\right)^{1/4}\sum_{k}h\left( \frac{4k}{K}\right)
\frac{1}{k^{11/2}}\ll \frac{l^{1/2+\epsilon}}{K^{11/2}}K.
\end{multline*}
Now we estimate the contribution of the term with $C_YZ_Y$, namely
\begin{equation*}
\frac{l^{\epsilon}}{\sqrt{l}}\sum_{n=1}^{l/2}\frac{(l/n)^{1/4}}{(\arcsin(\sqrt{n/l}))^{1/2}}\left|\sum_{k}h\left( \frac{4k}{K}\right)C_YZ_Y(\xi)\right|.
\end{equation*}
To this end, we use the series representation \eqref{cyexpseries} for $C_Y$ with a sufficiently large $n$ so that the error term is negligible. All main terms can be estimated in the same way and the largest contribution comes from the first summand  $-2\pi$.

The function $Z_Y(\xi)$ is defined by \eqref{zyxi}. Uisng \eqref{zyxi2} the error term is majorized by
\begin{equation*}
\epsilon_{3,1}(u,\xi)\ll \frac{\sqrt{\xi}|Y_0(u\sqrt{\xi})|}{u^3},
\end{equation*}
and therefore, its contribution is bounded by $$\frac{l^{1/2+\epsilon}}{K^{7/2}}K.$$

On the interval $(0, \pi^2/4)$ the functions $B_0(\xi)$, $A_1(\xi)$  are bounded, independent of $k$ and non-oscillatory (see  \eqref{eq:bo}, \eqref{eq:a1}).

The $Y$-Bessel functions have oscillatory behavior.
According to the equation \cite[Eq.~8.451(2)]{GR} we have
\begin{multline*}
Y_v(z)=\sqrt{\frac{2}{\pi z}}\sin\left( z-\frac{\pi v}{2}-\frac{\pi}{4}\right)
\Biggl(\sum_{s=0}^{s_1-1}
\frac{(-1)^s\Gamma(v+2s+1/2)}{(2z)^{2s}(2s)!\Gamma(v-2s+1/2)}+R_1\Biggr)\\+
\sqrt{\frac{2}{\pi z}}\cos\left( z-\frac{\pi v}{2}-\frac{\pi}{4}\right) \Biggl(\sum_{s=0}^{s_1-1}
\frac{(-1)^s\Gamma(v+2s+3/2)}{(2z)^{2s+1}(2s+1)!\Gamma(v-2s-1/2)}+R_2\Biggr),
\end{multline*}
 where $R_1=O(z^{-2s_1})$, $R_2=O(z^{-2s_1-1})$ by \cite[Eq.~8.451(7,8)]{GR}.
 Since $u\sqrt{\xi}>1$ the only difference between $Y_0(u\sqrt{\xi})$ and $Y_1(u\sqrt{\xi})$ is the shift on $\pi/2$ in the oscillating multiples. Thus it is sufficient to consider only $Y_0(u\sqrt{\xi})$.

The contribution of $R_1$, $R_2$ is majorized by
\begin{equation*}
\frac{l^{\epsilon}}{\sqrt{l}}\sum_{n=1}^{l/2}\sum_{k}\frac{h(4k/K)}{(k\sqrt{n/l})^{2s_1+1/2}}\ll
\frac{l^{\epsilon}}{\sqrt{l}}\left( \frac{l}{K^2}\right)^{s_1+1/4}K \quad \text{ for } s_1\geq 2.
\end{equation*}

It is sufficient to estimate
\begin{equation*}
E:=\frac{l^{\epsilon}}{\sqrt{l}}\sum_{n=1}^{l/2}\left( \frac{l}{n}\right)^{1/4}\left| \sum_{k}h\left(\frac{4k}{K}\right)
\frac{\sin\left( (4k-1)\arcsin(\sqrt{n/l})-\pi/4\right)}{\sqrt{4k-1}}\right|.
\end{equation*}
Using the Poisson summation formula (\cite[Theorem~4.4]{IK})
\begin{equation*}
\sum_{k}h\left(\frac{4k}{K}\right)
\frac{\sin\left( (4k-1)\arcsin(\sqrt{n/l})-\pi/4\right)}{\sqrt{4k-1}}=
\sum_{m \in \Z}I(m),
\end{equation*}
where
\begin{equation*}
I(m):=\int_{-\infty}^{+\infty}h\left( \frac{4y}{K}\right)\frac{\sin\left( (4y-1)\arcsin(\sqrt{n/l})-\pi/4\right)}{\sqrt{4y-1}}e^{-2\pi i ym}dy.
\end{equation*}
Let $g(y):=\frac{1}{4}Ky(-2\pi m\pm 4\arcsin(\sqrt{n/l}))$, then writing the sine in terms of exponentials we have
\begin{equation*}
I(m)\ll K\left|\int_{-\infty}^{\infty}h(y)e^{ig(y)}\frac{dy}{\sqrt{yK-1}} \right|.
\end{equation*}
Integration by parts $a\geq 2$ times yields
\begin{equation*}
I(m)\ll \sqrt{K}\left| \int_{-\infty}^{+\infty}\frac{\partial^a}{\partial y^a}\biggl(\frac{h(y)}{\sqrt{y-1/K}} \biggr) \frac{e^{ig(y)}dy}{K^a\left|-\pi m/2\pm \arcsin(\sqrt{n/l})\right|^a}\right|.
\end{equation*}
Note that $0< \arcsin(\sqrt{n/l})\leq \pi/4$.
If $m \neq 0$ we have
\begin{equation*}
I(m)\ll \sqrt{K}(Km)^{-a}
\end{equation*}
and
\begin{equation*}
I(0)\ll \sqrt{K}(K\arcsin(\sqrt{n/l}))^{-a}.
\end{equation*}
Consequently, we show that
\begin{multline*}
E \ll l^{-1/2+\epsilon}K^{1/2-a}\sum_{n=1}^{l/2}\left(\frac{l}{n} \right)^{1/4}\left(1+\frac{1}{(\arcsin(\sqrt{n/l}))^a} \right)\ll \\
l^{-1/2+\epsilon}K^{1/2-a}\sum_{n=1}^{l/2}\left(\frac{l}{n} \right)^{1/4+a/2}\ll  l^{a/2-1/4+\epsilon}
K^{-a+1/2}.
\end{multline*}
\end{proof}

Finally, we obtain the main result of this section.
\begin{thm}\label{thm:averagedsecondmomentmollified}For any $\epsilon>0$, any $a\geq 2$, $l\ll K^2$ the following asymptotic formula holds
\begin{multline}
A_2(l)=\frac{2\tau(l)}{\sqrt{l}}\frac{HK}{4}(2\log{K}-\log{l}-2\log{8\pi}+2\gamma+2H_1/H)+\\
O\Biggl(K\biggl(\frac{l^{a/2-1/4+\epsilon}}{K^{a+1/2}}+\frac{l^{\epsilon}}{\sqrt{lK}}+\frac{l^{1/2+\epsilon}}{K^{7/2}} \biggr) \Biggr).
\end{multline}
\end{thm}

The asymptotic formula for the averaged first moment follows from Theorem \ref{firstmoment} .
\begin{thm}\label{averagedfirstmoment}
There exist $c_1,c_2>0$ such that for $l \ll K$ we have
\begin{equation}
A_1(l)=\frac{2}{\sqrt{l}}\frac{HK}{4}+O\left(\frac{K}{\sqrt{l}}\left(c_1\frac{l}{K} \right)^{c_2K} \right).
\end{equation}
\end{thm}
\begin{xrem}
By more careful calculations, the error term in the above formula can be improved, as shown in \cite[Theorem~1.1]{BF3}.
\end{xrem}

\section{Mollification and non-vanishing at the critical point}\label{section:mollification}

In this section we apply the technique of mollification in order to establish effective non-vanishing results. With this goal, we mainly use the methods of \cite{KM,KMV}.

\subsection{The choice of mollifier}
We choose a mollifier of the type (see \cite{IS,KMV})
\begin{equation}
M(f)=\sum_{m \leq M}x_m\lambda_f(m)m^{-1/2},
\end{equation}
where
\begin{equation}\label{eq:xm}
x_m=\frac{\mu(m)}{\rho(m)}P\left(\frac{\log{M/m}}{\log{M}} \right),
\end{equation}
\begin{equation}
\rho(m)=\prod_{p|m}(1+1/p), \quad P(x)=x^2.
\end{equation}
If $M$ is not an integer (\cite[Lemma~2.1]{KMV})
\begin{equation}\label{eq:p}
\delta_{m<M}P\left(\frac{\log{M/m}}{\log{M}} \right)=\frac{2}{2\pi i}(\log{M})^{-2}\int_{(3)}\frac{M^s}{m^s}\frac{ds}{s^3}.
\end{equation}
\begin{lem}Let $k \equiv 0 \pmod{2}$, $M=k^{\Delta}$. For any $\epsilon>0$ there is $k_0=k_0(\epsilon)$ such that for every $k \geq k_0$  the inequality
\begin{equation}\label{ineq:mollif}
\sum_{f \in H_{2k}(1)}^{h}M^2(f)\ll \log{M}
\end{equation}
holds for any $\Delta<1-\epsilon$.
\end{lem}
\begin{proof}
Consider
\begin{equation*}
\sum_{f \in H_{2k}(1)}^{h}M^2(f)=\sum_{m_1,m_2\leq M}\frac{x_{m_1}x_{m_2}}{\sqrt{m_1m_2}}\sum_{f \in H_{2k}(1)}^{h}\lambda_f(m_1)\lambda_f(m_2).
\end{equation*}
The inner sum can be estimated using the Petersson trace formula and \cite[Lemma~2.1]{RS}
\begin{equation*}
\sum_{f \in H_{2k}(1)}^{h}\lambda_f(m_1)\lambda_f(m_2)=\delta(m_1,m_2)+O(e^{-k})\text{ for }m_1m_2<\frac{k^2}{10^4}.
\end{equation*}
Note that $x_m \ll 1$, and, therefore,
\begin{equation*}
\sum_{f \in H_{2k}(1)}^{h}M^2(f)\ll \sum_{m\leq M}\frac{1}{m}\ll \log{M}.
\end{equation*}
\end{proof}

Averaging over $k$ we prove the following estimate.
\begin{lem}Let $M=K^{\Delta}$. For any $\epsilon>0$ there is $K_0=K_0(\epsilon)$ such that for every $K \geq K_0$  the inequality
\begin{equation}\label{ineq:mollif2}
\sum_{k}h\left(\frac{4k}{K} \right)\sum_{f \in H_{4k}(1)}^{h}M^2(f)\ll K\log{M}
\end{equation}
holds for any $\Delta<1-\epsilon$.
\end{lem}

\subsection{The first mollified moment}

\begin{lem}\label{firstmol} Let $k\equiv 0 \pmod{2}$ and $M=k^{\Delta}$. For any $\epsilon>0$ there is $k_0=k_0(\epsilon)$ such that for every $k \geq k_0$ the following asymptotic formula holds
\begin{equation}\label{mollified1}
M_1:=\sum_{f \in H_{2k}(1)}^{h}M(f)L_f(1/2)=\frac{4\zeta(2)}{\log{M}}+O((\log{M})^{-2})
\end{equation}
for any $\Delta<1-\epsilon$.
\end{lem}

\begin{proof}
By Theorem \ref{firstmoment}
\begin{equation*}
\sum_{f \in H_{2k}(1)}^{h}M(f)L_f(1/2)=2\sum_{m\leq M}\frac{x_m}{m}+O\left(\sum_{m \leq M}\frac{x_m}{m}\left( \frac{2\pi e m}{k}\right)^k\right).
\end{equation*}
Since
\begin{equation*}
\sum_{m \leq M}\frac{x_m}{m}\left( \frac{2\pi em}{k}\right)^k\ll (2\pi eM/k)^k=(2\pi e)^k k^{k(\Delta-1)},
\end{equation*}
the error term is negligible for any $\Delta<1-\epsilon$.
Applying the equations \eqref{eq:xm} and \eqref{eq:p}, we have
\begin{equation*}
2\sum_{m\leq M}\frac{x_m}{m}=\frac{4}{(\log{M})^2}\frac{1}{2\pi i}\int_{(3)}\frac{M^s}{s^3}\sum_{m=1}^{\infty}\frac{\mu(m)}{\rho(m)m^{1+s}}ds.
\end{equation*}
Consider the sum over $m$:
\begin{equation*}
\sum_{m=1}^{\infty}\frac{\mu(m)}{\rho(m)m^{1+s}}=\frac{\alpha(s)}{\zeta(s+1)},
\end{equation*}
where
\begin{equation*}
\alpha(s)=\prod_p \frac{1+1/p-1/p^{s+1}}{(1+1/p)(1-1/p^{s+1})}
\end{equation*}
converges absolutely for $\Re{s}>-1$ and $\alpha(0)=\zeta(2)$.
The resulting integral
\begin{equation*}
\frac{4}{(\log{M})^2}\frac{1}{2\pi i}\int_{(3)}\frac{M^s\alpha(s)}{s^3\zeta(s+1)}ds
\end{equation*}
has a double pole at $s=0$.
We cross this pole by moving the contour of integration to
\begin{equation*}
\Re{s}=-\frac{c}{\log{(3+|\Im{s}|)}},
\end{equation*}
where $c>0$ is a constant sufficiently small so that there is no zero of $\zeta(s+1)$ to the right of the contour. Then the integral is bounded by
\begin{equation*}
\int_{0}^{\infty}M^{-c/\log{(3+t)}}(3+t)^{-3+\epsilon}dt=\int_{\log{3}}^{\infty}e^{x(-2+\epsilon)-cx^{-1}\log{M}}dx.
\end{equation*}
 Using the saddle point method, we estimate the last integral as
\begin{equation*}
(\log{M})^{1/4}e^{-c'\sqrt{\log{M}}}.
\end{equation*}
Finally, the residue at $s=0$ is equal to
\begin{equation*}
\frac{4\zeta(2)}{\log{M}}+O((\log{M})^{-2}).
\end{equation*}
\end{proof}

\begin{lem}\label{firstmolav} Let  $M=K^{\Delta}$. For any $\epsilon>0$ there is $K_0=K_0(\epsilon)$ such that for every $K \geq K_0$ one has
\begin{equation}\label{mollifieda1}
A_1:=\sum_{k}h\left(\frac{4k}{K} \right)\sum_{f \in H_{4k}(1)}^{h}M(f)L_f(1/2)=\frac{HK}{4}\frac{4\zeta(2)}{\log{M}}+O(K(\log{M})^{-2})
\end{equation}
for any $\Delta<1-\epsilon$.
\end{lem}
\begin{proof}
By Theorem \ref{averagedfirstmoment} for some absolute constants $c_1,c_2>0$
\begin{equation*}
A_1=\frac{HK}{4}2\sum_{m\leq M}\frac{x_m}{m}+O\Biggl( K\left(\frac{c_1M}{K}\right)^{c_2K}\Biggr).
\end{equation*}
The error term is negligible for any $\Delta<1-\epsilon$. The main term was evaluated in Lemma \ref{firstmol} .
\end{proof}

\subsection{The  second mollified moment}
\begin{lem}\label{lem:mol2} Let $k\equiv 0 \pmod{2}$ and $M=k^{\Delta}$. For  any $\epsilon>0$
there is $k_0=k_0(\epsilon)$ such that for every $k \geq k_0$ the following asymptotic formula holds
\begin{equation}\label{mollified2}
M_2:=\sum_{f \in H_{2k}(1)}^{h}M^2(f)L_{f}^{2}(1/2)=\frac{16\zeta^2(2)}{(\log{M})^2}(1+1/\Delta)+O((\log{M})^{-3})
\end{equation}
for any $\Delta<1/4-\epsilon$.
\end{lem}
\begin{proof}
Using the property of multiplicativity \eqref{eq:mult} we have
\begin{equation*}
\sum_{f \in H_{2k}(1)}^{h}M^2(f)L_{f}^{2}(1/2)= \sum_{b\leq M}\frac{1}{b}\sum_{m_1,m_2\leq M/b}\frac{x_{m_1b}x_{m_2b}}{\sqrt{m_1m_2}}M_2(m_1m_2).
\end{equation*}
By Theorem \ref{secondmoment} the contribution of the error term in $M_2(m_1m_2)$ is negligible for any $\Delta<1/4-\epsilon$. Indeed,
\begin{equation*}
\sum_{b\leq M}\frac{1}{b}\sum_{m_1,m_2\leq M/b}\frac{(m_1m_2)^{1/2+\epsilon}}{\sqrt{m_1m_2k}}\ll \frac{M^{2+\epsilon}}{\sqrt{k}}=k^{2\Delta-1/2+\epsilon}.
\end{equation*}
The main term of $M_2(m_1m_2)$ is
\begin{equation*}
2\frac{\tau{(m_1m_2)}}{\sqrt{m_1m_2}}\left(2\frac{\Gamma'}{\Gamma}(k)-\log{(m_1m_2)}+2\gamma-2\log{(2\pi)}\right).
\end{equation*}
Therefore, the largest contribution comes from
\begin{equation*}
2\frac{\tau{(m_1m_2)}}{\sqrt{m_1m_2}}\log{\frac{k^2}{m_1m_2}}.
\end{equation*}
Therefore, we need to compute
\begin{equation*}
2\sum_{b\leq M}\frac{1}{b}\sum_{m_1,m_2\leq M/b}\frac{\tau(m_1m_2)x_{m_1b}x_{m_2b}}{m_1m_2}\log{\frac{k^2}{m_1m_2}}.
\end{equation*}
Using
\begin{equation*}
\tau(m_1m_2)=\sum_{d|(m_1,m_2)}\mu(d)\tau(m_1/d)\tau(m_2/d),
\end{equation*}
we have
\begin{equation*}
2\sum_{n\leq M}\frac{1}{n}\sum_{d|n}\frac{\mu(d)}{d}\sum_{m_1,m_2\leq M/n}\frac{\tau(m_1)\tau(m_2)x_{m_1n}x_{m_2n}}{m_1m_2}\log{\frac{k^2}{d^2m_1m_2}}.
\end{equation*}
The last expression splits into two parts:
\begin{equation*}
P_1:=4\sum_{n\leq M}\frac{1}{n}\sum_{d|n}\frac{\mu(d)}{d}\sum_{m_1,m_2\leq M/n}\frac{\tau(m_1)\tau(m_2)x_{m_1n}x_{m_2n}}{m_1m_2}\log{k/m_1},
\end{equation*}

\begin{equation*}
P_2:=-4\sum_{n\leq M}\frac{1}{n}\sum_{d|n}\frac{\mu(d)\log{d}}{d}\sum_{m_1,m_2\leq M/n}\frac{\tau(m_1)\tau(m_2)x_{m_1n}x_{m_2n}}{m_1m_2}.
\end{equation*}

The main contribution comes from $P_1$ due to the additional factor of $\log{k/m_1}$. By Cauchy's integral formula
\begin{equation*}
\log{k/m_1}=\frac{1}{2\pi i}\int_{C_{\delta}}\frac{k^z}{m_{1}^z z^2}dz,
\end{equation*}
where $C_{\delta}$ is a circle of radius $\delta$ around the point $0$.
Using \eqref{eq:xm} and \eqref{eq:p}, we have
\begin{multline*}
P_1=\frac{16}{(\log{M})^4}\frac{1}{(2\pi i)^3}\int_{(3)}\int_{(3)}\int_{C_{\delta}}\sum_{n=1}^{\infty}\frac{\phi(n)}{n^{2+s_1+s_2}}M^{s_1+s_2}\times \\
\sum_{m_1,m_2=1}^{\infty}\frac{\tau(m_1)\tau(m_2)\mu(m_1n)\mu(m_2n)}{\rho(m_1n)\rho(m_2n)m_{1}^{s_1+z+1}m_{2}^{s_2+1}}\frac{k^zdz}{z^2}\frac{ds_1}{s_{1}^{3}}\frac{ds_2}{s_{2}^{3}}.
\end{multline*}
Let
\begin{equation*}
\alpha_1(s):=\prod_{p}\frac{(1+1/p-2p^{-s-1})(1-1/p)}{(1-p^{-s-1})^2},
\end{equation*}
and
\begin{equation*}
\beta_n(s):=\prod_{p|n}\frac{1+1/p}{1+1/p-2p^{-s-1}}.
\end{equation*}
Then the sum over $m_1$ can be computed as follows
\begin{multline*}
\sum_{m_1=1}^{\infty}\frac{\tau(m_1)\mu(m_1n)}{\rho(m_1n)m_{1}^{s_1+z+1}}=\frac{\mu(n)}{\rho(n)}\sum_{(m_1,n)=1}\frac{\tau(m_1)\mu(m_1)}{\rho(m_1)m_{1}^{s_1+z+1}}=\\
\frac{\mu(n)}{\rho(n)}\prod_{(p,n)=1}\left( 1-\frac{2}{p^{s_1+z+1}(1+1/p)} \right)= \frac{\mu(n)}{\rho(n)}\frac{\zeta(2)}{\zeta^2(s_1+z+1)}\beta_n(s_1+z)\alpha_1(s_1+z).
\end{multline*}
Similarly,
\begin{equation*}
\sum_{m_2=1}^{\infty}\frac{\tau(m_2)\mu(m_2n)}{\rho(m_2n)m_{2}^{s_2+1}}=\frac{\mu(n)}{\rho(n)}\frac{\zeta(2)}{\zeta^2(s_2+1)}\beta_n(s_2)\alpha_1(s_2).
\end{equation*}
Now the sum over $n$ is equal to
\begin{multline*}
\sum_{n=1}^{\infty}\frac{\phi(n)}{n^{2+s_1+s_2}}\frac{\mu^2(n)}{\rho^2(n)}\beta_n(s_1+z)\beta_n(s_2)=\\
\sum_{n=1}^{\infty}\frac{\phi(n)\mu^2(n)}{n^{2+s_1+s_2}}\prod_{p|n}(1+1/p-2p^{-s_1-z-1})^{-1}(1+1/p-2p^{-s_2-1})^{-1}=\\
\prod_p\left( 1+ \frac{(1-\frac{1}{p})(1+\frac{1}{p}-2p^{-s_1-z-1})^{-1} }{p^{s_1+s_2+1}(1+\frac{1}{p}-2p^{-s_2-1})}\right)=\zeta(s_1+s_2+1)\alpha_2(s_1,s_2,z),
\end{multline*}
where
\begin{equation*}
\alpha_2(s_1,s_2,z)=\prod_p \left( 1+ \frac{(1-\frac{1}{p})(1+\frac{1}{p}-\frac{2}{p^{s_{1}+z+1}})^{-1} }{p^{s_1+s_2+1}(1+\frac{1}{p}-\frac{2}{p^{s_{2}+1}})}\right)\left(1-\frac{1}{p^{s_1+s_2+1}} \right).
\end{equation*}
Let us denote
\begin{equation*}
\alpha(s_1,s_2,z):=\alpha_1(s_1+z)\alpha_1(s_2)\alpha_2(s_1,s_2,z).
\end{equation*}
Note that $\alpha(s_1,s_2,z)$ converges absolutely for $\Re{s_1}, \Re{s_2}, \Re{z}>-\epsilon$ for some $\epsilon>0$ and $\alpha(0,0,0)=1$.
As a result,
\begin{multline*}
P_1=\frac{16\zeta^2(2)}{(\log{M})^4}\frac{1}{(2\pi i)^3}\int_{(3)}\int_{(3)}\int_{C_{\delta}}M^{s_1+s_2}\frac{\alpha(s_1,s_2,z)\zeta(s_1+s_2+1)}{\zeta^2(s_1+z+1)\zeta^2(s_2+1)}
\frac{k^zdz}{z^2}\frac{ds_1}{s_{1}^{3}}\frac{ds_2}{s_{2}^{3}}.
\end{multline*}
Let $\gamma_i$ denote the contour
\begin{equation*}
\Re{s_i}=-\frac{c}{\log{(3+|\Im{s_i}|)}},
\end{equation*}
\begin{equation*}
f(s_1,s_2,z):=M^{s_1+s_2}\alpha(s_1,s_2,z) \frac{\zeta(s_1+s_2+1)}{\zeta^2(s_1+z+1)\zeta^2(s_2+1)}
\frac{k^z}{z^2s_{1}^{3}s_{2}^{3}}.
\end{equation*}
We start by evaluating the integral over $z$
\begin{multline*}
I:=\frac{1}{(2\pi i)^3}\int_{(3)}\int_{(3)}\int_{C_{\delta}}f(s_1,s_2,z)dzds_1ds_2=\\
\frac{1}{(2\pi i)^2}\int_{(3)}\int_{(3)}\res_{z=0}f(s_1,s_2,z)ds_2ds_1.
\end{multline*}
Further, we move the contours of integration to $(-\infty,-\delta)\cup c_{\delta}\cup(\delta,\infty)$,  where $\delta>0$ is a small positive number and $c_{\delta}$ is a semicircle in the right half plane. Note that the function $f(s_1,s_2,z)$ has not only a pole at $s_1=0$ but also a pole at $s_1+s_2=0.$ Moving the contour of integration to the line $\gamma_1$, we cross a pole at $s_1=0$. 
Accordingly, by \cite[Corollary 2.4.2, p. 55]{CMR} 
\begin{multline*}
I=\frac{1}{2\pi i}\int_{(3)}\res_{\substack{s_1=0\\z=0}}f(s_1,s_2,z)ds_2+\frac{1}{2\pi i} \int_{\gamma_1}\res_{\substack{s_2=-s_1\\z=0}}f(s_1,s_2,z)ds_1\\
+\frac{1}{(2\pi i)^2} \int_{\gamma_1}\int_{(3)} \res_{z=0}f(s_1,s_2,z)ds_2ds_1.
\end{multline*}
Then, as a next step, we move the contour of integration over $s_2$ to the line $\gamma_2$, getting
\begin{multline*}
I=\frac{1}{2\pi i} \int_{\gamma_1}\res_{\substack{s_2=0\\z=0}}f(s_1,s_2,z)ds_1+\frac{1}{2\pi i} \int_{\gamma_2}\res_{\substack{s_1=0\\z=0}}f(s_1,s_2,z)ds_2+\\
\frac{1}{(2\pi i)^2} \int_{\gamma_1}\int_{\gamma_2}\res_{z=0}f(s_1,s_2,z)ds_2ds_1+\\
\frac{1}{2\pi i} \int_{\gamma_1}\res_{\substack{s_2=-s_1\\z=0}}f(s_1,s_2,z)ds_1+\res_{s_1=s_2=z=0}f(s_1,s_2,z).
\end{multline*}
The contribution of the first three integrals above is negligible and can be estimated similarly to the proof of Lemma \ref{firstmol}. The fourth integral can be bounded by a constant and, therefore, its contribution to $P_1$ is
$O((\log{M})^{-4})$.
The main term is given by the residue at $s_1=s_2=z=0$.
The function $f(s_1,s_2,z)$ has a simple pole at $s_2=0$. Hence
\begin{equation*}
\res_{s_1=s_2=z=0}f(s_1,s_2,z)=2\pi i\res_{s_1=z=0}\frac{k^zM^{s_1}\zeta(s_1+1)\alpha(s_1,0,z)}{z^2s_{1}^{3}\zeta^2(s_1+z+1)}.
\end{equation*}
Next, we compute the residue at $z=0$, where the resulting function has a double pole. Finally, evaluating the residue at the triple pole $s_1=0$ we find that
\begin{multline*}
P_1=\frac{16\zeta^2(2)}{(\log{M})^4}(\log{k}\log{M}+(\log{M})^2)+O((\log{M})^{-3})=\\\frac{16\zeta^2(2)}{(\log{M})^2}(\Delta^{-1}+1)+O((\log{M})^{-3}).
\end{multline*}
Similarly, using the representation
\begin{equation*}
\log{d}=\frac{1}{2\pi i}\int_{C_{\delta}}\frac{d^z}{ z^2}dz,
\end{equation*}
we prove that $P_2=O((\log{M})^{-3}).$
\end{proof}

\begin{lem}\label{lem:molaveraged2} Let $M=K^{\Delta}$. For  any $\epsilon>0$
there is $K_0=K_0(\epsilon)$ such that for every $K \geq K_0$ the following asymptotic formula holds
\begin{multline}\label{mollified2}
A_2:=\sum_{k}h\left(\frac{4k}{K}\right)\sum_{f \in H_{4k}(1)}^{h}M^2(f)L_{f}^{2}(1/2)=\\
\frac{HK}{4}\frac{16\zeta^2(2)}{(\log{M})^2}(1+1/\Delta)+O(K(\log{M})^{-3})
\end{multline}
for any $\Delta<1-\epsilon$.
\end{lem}
\begin{proof}
Consider
\begin{equation*}
A_2=\sum_{b\leq M}\frac{1}{b}\sum_{m_1,m_2\leq M/b}\frac{x_{m_1b}x_{m_2b}}{\sqrt{m_1m_2}}A_2(m_1m_2),
\end{equation*}
where the asymptotics of $A_2(m_1m_2)$ is given by Theorem \ref{thm:averagedsecondmomentmollified}. Accordingly, the contribution of the error term  is bounded by
\begin{equation*}
KM^{\epsilon}\left( \frac{M^{a+1/2}}{K^{a+1/2}}+\frac{1}{\sqrt{K}}+\frac{M^2}{K^{7/2}}\right)
\end{equation*}
for any $a \geq 2$.
This is negligible if $\Delta<1-\epsilon$.
The main term can be evaluated similarly to Lemma \ref{lem:mol2}.
\end{proof}


\subsection{Non-vanishing for the individual weight}

\begin{thm}\label{non-vanishing harmonic theorem} For any $\epsilon>0$ there exists $k_0=k_0(\epsilon)$ such that for any $k\geq k_0$ and $k \equiv 0\Mod{2}$ we have
\begin{equation}
\sum_{\substack{f \in H_{2k}(1)\\ L_{f}(1/2)\geq (\log{k})^{-2}}}^{h}1 \geq \frac{1}{5}-\epsilon.
\end{equation}
\end{thm}
\begin{proof}
Asymptotic formulas for the first and second mollified moments  are given by the equations \eqref{mollified1} and \eqref{mollified2}. Accordingly, the largest admissible length of mollifier is $\Delta<1/4-\epsilon$.
Applying the inequality \eqref{ineq:mollif}, we estimate
\begin{multline*}
\widetilde{M}_1:=\sum_{f \in H_{2k}(1)}^{h}M(f)L_f(1/2)\delta_{L_f(1/2)<b(k)(\log{k})^{-1/2}}\leq \\
\left(\sum_{\substack{f \in H_{2k}(1)\\L_f(1/2)<b(k)(\log{k})^{-1/2}}}^{h}L_{f}^{2}(1/2) \right)^{1/2}
\left(\sum_{f \in H_{2k}(1)}^{h} M^2(f)\right)^{1/2}\leq\\
b(k)(\log{k})^{-1/2}\left(\sum_{f \in H_{2k}(1)}^{h} M^2(f)\right)^{1/2}\leq b(k).
\end{multline*}
Taking $b(k)=(\log{k})^{-3/2}$  we have
\begin{equation*}
\sum_{\substack{f \in H_{2k}(1)\\ L_{f}(1/2)\geq (\log{k})^{-2}}}^{h}1 \geq \frac{(M_{1}-\widetilde{M}_1)^{2}}{M_{2}}\geq \frac{\Delta}{1+\Delta}
\end{equation*}
for any $\Delta<1/4-\epsilon$. The result follows.
\end{proof}

\subsection{Non-vanishing on average}

\begin{thm}\label{thm: on average} For any $\epsilon>0$ there is $K_0=K_0(\epsilon)$ such that for any $K\geq K_0$ we have
\begin{equation}
\frac{4}{HK}\sum_{k}h\left(\frac{4k}{K} \right)\sum_{\substack{f \in H_{4k}(1)\\ L_{f}(1/2)\geq (\log{k})^{-2}}}^{h}1
\geq \frac{1}{2}-\epsilon.
\end{equation}
\end{thm}
\begin{proof}
Note that
\begin{equation*}
\sum_{k}h\left(\frac{4k}{K} \right)\sim \frac{HK}{4} \text{ as } K\rightarrow \infty.
\end{equation*}
The Cauchy-Schwartz inequality and the estimate \eqref{ineq:mollif2} yield
\begin{equation*}
\widetilde{A}_1:=\sum_{k}h\left(\frac{4k}{K} \right)\sum_{\substack{f \in H_{4k}(1)\\L_{f}(1/2)<b(k) (\log{k})^{-1/2}}}^{h}M(f)L_f(1/2)\ll Kb(k).
\end{equation*}
Choosing $b(k)=(\log{k})^{-3/2}$ and applying the Cauchy-Schwartz inequality twice, we obtain
\begin{multline*}
\sum_{k}h\left(\frac{4k}{K} \right)\sum_{f \in H_{4k}(1)}^{h}L_f(1/2)\delta_{L_{f}(1/2)\geq (\log{k})^{-2}}\leq\\
\sum_{k}h\left(\frac{4k}{K} \right)\sqrt{\sum_{f \in H_{4k}(1)}^{h}L_{f}^{2}(1/2)}\sqrt{\sum_{f \in H_{4k}(1)}^{h}\delta_{L_{f}(1/2)\geq (\log{k})^{-2}}}\leq\\
\left(\sum_{k}h\left(\frac{4k}{K} \right)\sum_{f \in H_{4k}(1)}^{h}L_{f}^{2}(1/2)\right)^{1/2} \left(\sum_{k}h\left(\frac{4k}{K} \right)\sum_{\substack{f \in H_{4k}(1)\\L_{f}(1/2)\geq (\log{k})^{-2}}}^{h}1\right)^{1/2}.
\end{multline*}

Therefore, by Lemmas \ref{firstmolav}  and \ref{lem:molaveraged2} we have
\begin{equation*}
\frac{4}{HK}\sum_{k}h\left(\frac{4k}{K} \right)\sum_{\substack{f \in H_{4k}(1)\\ L_{f}(1/2)\geq (\log{k})^{-2}}}^{h}1 \geq \frac{\left(\frac{4}{HK}A_{1}-\frac{4}{HK}\widetilde{A}_1\right)^2}{\frac{4}{HK}A_{2}} \geq \frac{\Delta}{1+\Delta}
\end{equation*}
for any $\Delta<1-\epsilon$.
\end{proof}

\subsection{Removing the harmonic weight}
In order to state Theorem \ref{non-vanishing harmonic theorem} for the natural average we apply the techniques developed by Kowalski and Michel in \cite{K} and \cite{KM}. 
\begin{lem}\label{from harmonic average to natural}
Let $\alpha_f$ be a sequence of complex numbers such that
\begin{equation}\label{first condition}
\sum_{f \in H_{2k}(1)}^{h}|\alpha_f|\ll (\log k)^A,\quad \hbox{for some} \quad A>0,
\end{equation}
\begin{equation}\label{second condition}
\max_{f \in H_{2k}(1)}|\omega_f\alpha_f|\ll k^{-\delta},\quad \hbox{for some} \quad \delta>0.
\end{equation}
Let $x=k^{\epsilon}$ and
\begin{equation}\label{shortdirichletpolyn}
\omega_f(x):=\sum_{n\le x}\frac{\rho_f(n)}{n}=
\sum_{dl^2\le x}\frac{\lambda_f(d^2)}{dl^2},
\end{equation}
 where $\rho_f(n)$ is defined by \eqref{sym2 def}. Then there exists $\kappa=\kappa(\epsilon,\delta)>0$ such that
\begin{equation}
\sum_{f \in H_{2k}(1)}\alpha_f=\frac{|H_{2k}(1)|}{\zeta(2)}\sum_{f \in H_{2k}(1)}^{h}\omega_f(x)\alpha_f+O(k^{1-\kappa}).
\end{equation}
\end{lem}
\begin{proof}
Combining the formula  $$|H_{2k}(1)|=(2k-1)/12+O(1)$$ with  \eqref{harmonic weight} and the bound of Hoffstein-Lockhart \cite{HL} on $L(\sym^2f,1)$, we conclude  that
\begin{equation}\label{harmonic weight2}
\frac{1}{\omega_f}=\frac{L(\sym^2f,1)}{\zeta(2)}|H_{2k}(1)|+O((\log k)^3).
\end{equation}
Using \eqref{harmonic weight2} and  \eqref{first condition} we obtain
\begin{equation}\label{natural weight1}
\sum_{f \in H_{2k}(1)}\alpha_f=
\sum_{f \in H_{2k}(1)}^{h}\frac{\alpha_f}{\omega_f}=
\frac{|H_{2k}(1)|}{\zeta(2)}\sum_{f \in H_{2k}(1)}^{h}\alpha_fL(\sym^2f,1)+O(k^{1-\kappa}).
\end{equation}
Now the key idea is to replace $L(\sym^2f,1)$ by a short Dirichlet polynomial $\omega_f(x)$ defined by \eqref{shortdirichletpolyn}.
Let us also introduce
\begin{equation*}
\omega_f(x,y):=\sum_{x<n\le y}\frac{\rho_f(n)}{n}=
\sum_{x<dl^2\le y}\frac{\lambda_f(d^2)}{dl^2}.
\end{equation*}
It follows from \cite[Lemmas 2.3]{LW} that for a sufficiently large constant $a$ (one can take, for example,  $a=10$) and $y=k^a$ the following asymptotic formula holds
\begin{equation}\label{sym2 approx}
L(\sym^2f,1)=\omega_f(x)+\omega_f(x,y)+O(k^{-1+\epsilon}).
\end{equation}
Substituting \eqref{sym2 approx} to \eqref{natural weight1} we obtain
\begin{equation}\label{natural weight2}
\sum_{f \in H_{2k}(1)}\alpha_f=
\frac{|H_{2k}(1)|}{\zeta(2)}\sum_{f \in H_{2k}(1)}^{h}\alpha_f\omega_f(x)+
\frac{|H_{2k}(1)|}{\zeta(2)}\sum_{f \in H_{2k}(1)}^{h}\alpha_f\omega_f(x,y)+O(k^{1-\kappa}).
\end{equation}
Repeating the arguments of \cite[Proposition 2]{KM}, applying \eqref{first condition} and \eqref{second condition}, and using instead of \cite[Lemma 3]{KM} its analogue in the weight aspect, namely \cite[Lemmas 2.5]{LW}, we obtain
\begin{equation*}
\frac{|H_{2k}(1)|}{\zeta(2)}\sum_{f \in H_{2k}(1)}^{h}\alpha_f\omega_f(x,y)\ll k^{1-\kappa}.
\end{equation*}
This completes the proof.
\end{proof}
Note that the main terms in the asymptotic formulas for the twisted moments in the weight aspect have only minor changes comparing with the main terms in the  level aspect. Thus  we can  follow closely the approach of \cite[Sec. 5]{KM2}. 

\begin{thm}\label{non-vanishing theorem}
For any $\epsilon>0$ there exists $k_0=k_0(\epsilon)$ such that for any $k\geq k_0$ and $k \equiv 0\Mod{2}$ we have
\begin{equation}
\frac{1}{|H_{2k}(1)|}\sum_{\substack{f \in H_{2k}(1)\\ L_{f}(1/2)\geq (\log{k})^{-2}}}1 \geq \frac{1}{5}-\epsilon.
\end{equation}
\end{thm}
\begin{proof}
The proof is similar to the one of Theorem \ref{non-vanishing harmonic theorem} and is based on the asymptotic formulas for the first and second mollified moments. Combining Theorems \ref{firstmoment} and \ref{secondmoment}, Lemma \ref{from harmonic average to natural}, and choosing the same mollifier as in \cite[Sec. 5]{KM2}, we obtain for $\Delta<1/4-\epsilon$
\begin{equation*}
\sum_{f \in H_{2k}(1)}M(f)L_{f}(1/2)=2\zeta(2)|H_{2k}(1)|\left(1+O(k^{-\epsilon_2})\right),
\end{equation*}
\begin{equation*}
\sum_{f \in H_{2k}(1)}M^2(f)L_{f}^2(1/2)=2\zeta^2(2)|H_{2k}(1)|\frac{\log k^2+2\log M}{\log M}
\left(1+O\left(\frac{\log\log k}{\log k}\right)\right),
\end{equation*}
where as usual $M=k^{\Delta}$ is the length of the mollifier and $\epsilon,\epsilon_2>0.$  Note that  Lemma \ref{from harmonic average to natural} can be applied for the first and second mollified moments only if  the conditions \eqref{first condition} and \eqref{second condition} are satisfied for $\alpha_f=M(f)L_{f}(1/2)$ and $\alpha_f=M^2(f)L_{f}^2(1/2)$. To show that  \eqref{first condition} holds we use the Cauchy-Schwarz inequality, the estimate \eqref{ineq:mollif} and Theorem \ref{secondmoment} in the case of the first moment. While in the case of the second moment we proceed as in Lemma \ref{lem:mol2}. To show that the condition \eqref{second condition} is satisfied we apply the subconvexity bound $L_{f}(1/2)\ll k^{1/3+\epsilon}$ due to Jutila-Motohashi \cite{JM} and Peng \cite{Peng}. This bound shows that the condition \eqref{second condition} is satisfied for any $\Delta<1/3.$
\end{proof}

\bigskip
\footnotesize
\noindent\textit{Acknowledgments.}
The authors thank Viktor A. Bykovskii and the Institute for Applied Mathematics in Khabarovsk for hospitality and excellent working conditions. We are grateful to Guillaume Ricotta for careful reading of an earlier draft and helpful comments. We thank Philippe Michel and Emmanuel Royer for encouraging discussions. Finally, we express our thanks to the referees for their extraordinary careful reading of this manuscript and many suggestions for improvement.

Research of O. Balkanova is supported by Academy of Finland project no. $293876$. 

Research of D. Frolenkov is supported by the Russian Science Foundation under grant [14-11-00335] and performed in Khabarovsk Division of the Institute for Applied Mathematics, Far Eastern Branch, Russian Academy of Sciences and is partially supported by the 
Foundation for the Advancement of Theoretical Physics and Mathematics "BASIS".


\nocite{}

\end{document}